\documentclass[12pt,reqno]{amsart}
\usepackage{latexsym,amsmath,amssymb, amsthm, mathscinet}
\usepackage{cases, verbatim}
\usepackage[dvipdfmx]{graphicx}
\usepackage{bmpsize}
\usepackage{float}

\usepackage{amsfonts,amsmath,amsthm, amssymb}
\usepackage{cases}
\usepackage[usenames]{color}
\usepackage{enumerate}
\usepackage{bm}

\theoremstyle{plain}
\newtheorem{thm}{Theorem}[section]

\newtheorem{lem}[thm]{Lemma}
\newtheorem{cor}[thm]{Corollary}
\newtheorem{prop}[thm]{Proposition}
\theoremstyle{remark}
\newtheorem{rem}{\bf{Remark}}
\numberwithin{equation}{section}

\newcommand{\N}{\mathbb{N}}

\newcommand{\R}{\mathbb{R}}

\newcommand{\T}{\mathbb{T}}
\newcommand{\Z}{\mathbb{Z}}

\newcommand{\cA}{\mathcal{A}}

\newcommand{\cC}{\mathcal{C}}

\newcommand{\AC}{{\rm AC\,}}

\newcommand{\BUC}{{\rm BUC\,}}
\newcommand{\AL}{{\rm AL\,}}

\newcommand{\Lip}{{\rm Lip\,}}

\newcommand{\gam}{\gamma}

\newcommand{\ep}{\varepsilon}

\newcommand{\ol}{\overline}

\begin{document}
\baselineskip=15pt

\title[Quantitative homogenization of convex HJ equations]
{Quantitative homogenization of\\ 
convex Hamilton--Jacobi equations with\\ 
$u/\ep$-periodic Hamiltonians}

\author{Hiroyoshi Mitake}
\address[H. Mitake]{
	Graduate School of Mathematical Sciences,
	University of Tokyo
	3-8-1 Komaba, Meguro-ku, Tokyo, 153-8914, Japan}
\email{mitake@g.ecc.u-tokyo.ac.jp}

\author{Panrui Ni}
\address[P. Ni]{
Department 1: Graduate School of Mathematical Sciences, University of Tokyo, 3-8-1 Komaba, Meguro-ku, Tokyo 153-8914, Japan; Department 2: Shanghai Center for Mathematical Sciences, Fudan University, Shanghai 200438, China}
\email{panruini@g.ecc.u-tokyo.ac.jp}

\author{Hung V. Tran}
\address[H. V. Tran]{
Department of Mathematics, University of Wisconsin Madison, Van Vleck Hall, 480 Lincoln Drive, Madison, Wisconsin 53706, USA}
\email{hung@math.wisc.edu}

\makeatletter
\@namedef{subjclassname@2020}{\textup{2020} Mathematics Subject Classification}
\makeatother

\date{\today}
\keywords{Homogenization; Hamilton--Jacobi equations; convex Hamiltonians; dislocation dynamics; optimal convergence rates}
\subjclass[2020]{
	35B27, 
	35B40, 
        49L25 
}

\begin{abstract}
Here, we study quantitative homogenization of first-order convex Hamilton--Jacobi equations with $(u/\ep)$-periodic Hamiltonians which typically appear in dislocation dynamics. 
Firstly, we establish the optimal convergence rate by using the inherent fundamental solution and the implicit variational principle of Hamilton dynamics with their Hamiltonian depending on the unknown.
Secondly, under additional growth assumptions on the Hamiltonian, we establish global H\"older regularity for both the solutions and the correctors, serving as a notable application of our quantitative homogenization theory.
\end{abstract}

\date{\today}

\maketitle

\subjclass{ }


\section{Introduction}

In this paper, we investigate the homogenization for Hamilton--Jacobi equations depending rapidly and periodically on the unknown of the form: 
\begin{equation}\label{e}
\begin{cases}
u^\varepsilon_t+H(\frac{x}{\varepsilon},\frac{u^\varepsilon}{\varepsilon},Du^\varepsilon)=0
\quad &\text{for} \ x\in\R^n,\ t>0,
\\ 
u^\varepsilon(x,0)=\varphi(x)
\quad &\text{for} \ x\in\R^n, 
\end{cases}
\end{equation}
where 
$\ep>0$, $n\in\N$, and 
$H:\R^n\times\R\times\R^n\to\R$, 
and $\varphi:\R^n\to\R$ are given continuous functions. Here,  
$u^\ep:\R^n\times[0,\infty)\to\R$ is the unknown function. 
\textit{Throughout} the paper, we assume 
\begin{itemize}
\item[(H1)] 
   $H\in C(\R^n\times\R\times \R^n)$, and 
    $(y,r)\mapsto H(y,r,p)$ is $\Z^{n+1}$-periodic for all $p\in \R^n$, that is, 
\[
H(y+z,r+s,p)=H(y,r,p) \quad\text{for all} \  (y,r)\in\R^n\times\R, (z,s)\in\Z^n\times\Z, \ \text{and}  \ p\in\R^n; 
\] 
\item[(H2)] 
there exists a constant $K>0$ such that 
\[
|H(y,r,p)-H(y,s,p)|\le K|r-s|
\quad\text{ for all} \ (y,p)\in\R^n\times\R^n, \text{ and } r,s\in \R;
\] 
\item[(H3)] $H$ is superlinear in $p$, that is, 
\[
\lim_{|p|\to\infty}\left(\inf_{(y,r)\in\R^n\times\R}\frac{H(y,r,p)}{|p|}\right)=\infty; 
\]
\item[(H4)] $p\mapsto H(y,r,p)$ is convex for all $(y,r)\in\R^n\times\R$; 
\item[(H5)] $\varphi\in \Lip(\mathbb R^n)$. 
Here, for a given metric space $X$, we denote by  $\Lip(X)$ the set of Lipschitz continuous functions on $X$.
\end{itemize}
We are always concerned with viscosity solutions,
and the adjective ``viscosity" is often omitted in the paper. 
For each $T\ge 0$, denote by $\AL(\R^n\times [0,T])$ the set of functions $w:\R^n\times [0,T]\to\R$ that grow at most linearly in $x$.
That is, $w\in \AL(\R^n\times [0,T])$ if there exists a constant $C_T=C_T(w)>0$ such that 
\[
|w(x,t)|\leq C_T(1+|x|) \quad \text{ for $(x,t)\in \R^n\times [0,T]$.}
\]
Set
\[
\AL(\R^n\times [0,\infty)) = \left\{w:\R^n\times [0,\infty)\to\R \mid w\in \AL(\R^n\times [0,T]) \text{ for all } T>0 \right\}.
\]
We only care about the unique solution of \eqref{e} in $\AL(\R^n\times [0,\infty))$ (see Lemma \ref{lem:CP} for the corresponding comparison principle). If $\varphi$ is moreover bounded on $\R^n$, we consider the unique solution of \eqref{e} in
\[\BUC(\R^n\times [0,\infty)) = \left\{w:\R^n\times [0,\infty)\to\R \mid w\in \BUC(\R^n\times [0,T]) \text{ for all } T>0 \right\}.\]
Here, for a given metric space $X$, $\BUC(X)$ is the set of bounded and uniformly continuous functions on $X$.

In Appendix \ref{sec:moti}, we provide a derivation of the homogenization problem as a representative example inspired by dislocation dynamics, originally introduced by \cite{IM}.
According to \cite[Proposition 3.2]{CLQY}, the Hamilton-Jacobi equations periodically depending on the unknown function are related to the Sine-Gordon equation. We also refer the reader to related works on $u/\ep$-homogenization of the semilinear heat equation (see \cite{CDN}) and the scalar phase-field equation (see \cite{J}). Another motivation comes from a conjecture of De Giorgi, which concerns the homogenization of linear transport equations. This problem was originally raised in \cite{DG} (see the discussions in Section \ref{further}).

In \cite{IM, B}, under (H1), (H2), the coercivity of $H$ in $p$, and additionally Lipschitz continuity properties of $H$ in $y$ and $p$, that is, 
there exists $\gamma_0>0$ such that 
\begin{equation}\label{ass:IM}
|D_yH(y,r,p)|\le \gamma_0(1+|p|), \quad
|D_pH(y,r,p)|\le \gamma_0
\quad
\text{ for a.e. } (y,r,p)\in\R^n\times\R\times\R^n, 
\end{equation} 
the qualitative homogenization was studied. 
Moreover, in \cite{AP}, 
by adapting the technique introduced in \cite{CDI} to their setting, 
a rate of convergence $O(e^T \ep^{\frac{1}{3}})$ of $u^\ep$ to $u$ on $\R^n\times [0,T]$ was obtained for each $T>0$. 
Of course, this rate depends exponentially on $T$.

\smallskip
A primary goal of this paper is to establish the optimal convergence rate $O(\ep)$ for the homogenization problem \eqref{e} under (H1)--(H5). 
To get this result, we first establish the implicit variational problem shown in Theorem \ref{thm:fundamental-sol} below. It is the first time that the implicit variational principle is established for continuous and convex Lagrangians. An equivalent framework was established in \cite{CCJWY} for continuous and strictly convex Lagrangians. Let $m(t,x,y,c)$ be the function defined in Theorem \ref{thm:fundamental-sol}. Here is our first main result.
\begin{thm}\label{thm1}
Assume {\rm(H1)--(H5)}. 
For $(x,t)\in \R^n\times [0,\infty)$, set
\[u^\ep(x,t):=\inf_{y\in\R^n}\ep m\Big(\frac{t}{\ep},\frac{x}{\ep},\frac{y}{\ep},\frac{\varphi(y)}{\ep}\Big),\]
where $m:[0,\infty)\times\R^n\times\R^n\times\R\to\R$ is the \textit{fundamental solution} {\rm(}or the \textit{distance function}{\rm)} associated with \eqref{e} with $\ep=1$, which is given by Theorem {\rm\ref{thm:fundamental-sol}}. 
Then, $u^\ep$ is the unique solution of \eqref{e} in $\AL(\R^n\times [0,\infty))$. 
We have that the limit $u(x,t):=\lim_{\ep\to 0}u^\ep(x,t)$ exists for all $(x,t)\in\R^n\times[0,\infty)$. 
Moreover, there exists a constant $C>0$ depending only on $H$ and $\|D\varphi\|_{L^\infty(\mathbb R^n)}$ such that, for $\ep\in(0,1)$, 
\begin{equation}\label{eq:rate-thm}
\|u^\ep-u\|_{L^\infty(\mathbb R^n\times[0,\infty))}\le C\ep.
\end{equation}
The rate $O(\ep)$ in \eqref{eq:rate-thm} is optimal. 
\end{thm}

When $H=H(x,p)$, which is independent of $r$, the optimal convergence rate $O(\ep)$ was proved in \cite{TY} (see also \cite{TY-lecture}). 
Specifically, \cite[Theorem 1.1]{TY} can be considered as a particular case of Theorem \ref{thm1}, and hence, the rate $O(\ep)$ in \eqref{eq:rate-thm} is optimal. 
Since we assume the superlinearity of $H$ in $p$ in (H3), 
our settings are different from those in \cite{IM, B, AP}. 
In particular, $H$ cannot be globally Lipschitz in $p$ as in \eqref{ass:IM}.
Therefore, as far as the authors know, the convergence result in Theorem \ref{thm1} itself is new in the literature. 
We would like to mention that the quantitative analysis for \eqref{e} was addressed as an open problem in the lecture notes by Tran and Yu in \cite[Question 7, Section 8.4]{TY-lecture}. 
In contrast to \cite{AP}, it is surprising that the rate in \eqref{eq:rate-thm} can be obtained independently of time $t\in[0,\infty)$.

\smallskip
 
Quantitative homogenization theory for first-order convex Hamilton–Jacobi equations in the periodic setting has seen significant recent development. This progress was initiated in \cite{MTY, TY}, where the authors made essential use of the representation formula of the solutions from optimal control theory and key properties of minimizing curves.
In \cite{TY}, the authors obtained the subadditivity and superadditivity properties of the metric functions (or the fundamental solutions) associated with the Lagrangians. 
In particular, the method to obtain the superadditivity property is innovative using a ``curve surgery" argument. 
Loosely speaking, the curve surgery argument is 
a combination of three steps: 
(i) cutting the optimal trajectory of the metric function by using 
Burago's cutting lemma 
(see \cite[Lemma 2]{Burago}, and also \cite[Lemma 2.1]{TY}), 
(ii) 
shifting the cut curves periodically, 
and 
(iii) patching them back together in a skillful manner. 
This method is robust, and 
we refer the reader to 
\cite{HJ, HJMT, Hu-Tu-Zhang, MN1, MN2, NT, T} and the references therein for further development in this direction. 

\smallskip

In our problem \eqref{e}, since the Hamiltonian depends on the unknown function periodically, we can expect the implicit variational principle for the value function. 
However, performing a curve surgery argument requires great care.
In \cite{MN1}, a sharp convergence rate $O(\ep^{\frac{1}{2}})$ for 
$u_t^\ep+H(\frac{x}{\ep}, u^\ep, Du^\ep)=0$ was obtained 
under (H1), (H2), (H4), and the coercivity of $H$ in $p$ by using an iteration type argument.
However, this iteration does not work for our current problem because of the rapidly oscillating term $u/\ep$.  
Here, we essentially use the implicit variational principle of the Hamiltonian dynamics, and develop a curve surgery argument for inherent fundamental solutions. 
We begin by cutting curves in the $(x,u)$-space, rather than just in 
$x$-space, using Burago's cutting lemma. 
We then exploit the periodic structure of the Hamiltonian in $(x,u)$, and finally, we use a key idea inspired by the Herglotz variational principle to patch the curves back together.

We also point out several delicate differences and difficulties. 
First of all, we notice that it is not clear if the family $\{u^\ep\}_{\ep>0}$ is equi-Lipschitz continuous even locally in time. 
This is mainly because the Hamiltonian $H=H(y,r,p)$ does not satisfy the non-decreasing property in $r$. First, since $H(y,r,p)$ is periodic in both $y$ and $r$, we can see that
\[M:=\sup_{\substack{(y,r)\in\R^n\times\R,\\ p\in B(0, \|D\varphi\|_{L^\infty(\R^n)})}}|H(y,r,p)|<\infty,\]
and that $\varphi(x)\pm Mt$ are a sub/super solution to \eqref{e}, respectively. 
The trouble is that we cannot use the comparison principle to get the non-expansiveness property.
More precisely, we cannot expect that, for $t,s\geq 0$,
\[
\|u^\ep(\cdot,t+s)-u^\ep(\cdot,t)\|_{L^\infty(\R^n)}\leq \|u^\ep(\cdot,s)-u^\ep(\cdot,0)\|_{L^\infty(\R^n)}\leq Ms.
\] 
We only have
\[
\|u^\ep(\cdot,t+s)-u^\ep(\cdot,t)\|_{L^\infty(\R^n)}\leq  e^{Kt/\ep}\|u^\ep(\cdot,s)-u^\ep(\cdot,0)\|_{L^\infty(\R^n)}\leq e^{Kt/\ep}Ms,
\] 
where $K$ is given by (H2). See \eqref{eks} below. This immediately implies a rough bound 
\[
|u^\ep_t(x,t)| \leq e^{Kt/\ep}M \quad \text{ for a.e. } (x,t)\in \R^n \times [0,\infty).
\]
This rough bound depends on both $\ep$ and $t$ and blows up as $\ep \to 0$. 
Of course, this estimate holds for general Hamiltonian $H(y,r,p)$ which is $K$-Lipschitz continuous in $r$ uniform in $(y,p)\in\R^n\times\R^n$. We refer the reader to \cite[Lemma 2.3]{MN1} for the same estimate for weakly coupled systems of Hamilton--Jacobi equations with a very general coupling, and also \cite[Lemma 2.2 (ii)]{NWY2} for single Hamilton--Jacobi equations on compact manifolds. Therefore, if we only assume the coercivity of $H$, i.e., 
 $\lim_{\left|p\right| \to \infty}\inf_{(y,r) \in \R^n\times\R}H(y,r,p) = \infty$, 
we cannot modify $H$ to have a quadratic growth. 
For this reason, we assume $H$ to satisfy the superlinearity in $p$ as in (H3). 
We refer to \cite{NT} for a similar difficulty when we consider time-dependent Hamilton--Jacobi equations. 

\smallskip
Next, let $\ep v(\frac{x}{\ep},\frac{t}{\ep})=u^\ep(x,t)$ and $y=\frac{x}{\ep}$, $\tau=\frac{t}{\ep}$, then we have 
\[
v_\tau+H(y,v,D_yv)=0.
\]
The Lipschitz estimate of $u^\ep$ is equivalent to the Lipschitz estimate of $v$. 
In \cite[Theorem 1.7 (2)]{NWY}, the authors proved that $\|D_yv(\cdot,\tau)\|_{L^\infty(\R^n)}$ is bounded by a constant depending only on $H$ for all $\tau>1$, where $y$ belongs to a compact manifold without boundary, $H$ is of class $C^3$, $r\mapsto H(y,r,p)$ is periodic, and $p\mapsto H(y,r,p)$ is superlinear and strictly convex. The proof is based on the energy estimate of minimizers of the optimal control formula of $v$. 
See also \cite[Proposition 1.6]{NW} for a related result. 
Since $\varphi \in\Lip(\R^n)$ is not periodic in general, due to the lack of compactness, the arguments in \cite{NWY,NW} cannot be applied in our case here.

\smallskip
Now, it is worth emphasizing that although Theorem \ref{thm1} guarantees the uniform convergence of the solution $u^\ep$ of \eqref{e} 
to the limit function $u$ as $\ep\to0$ under (H1)--(H5), 
it is not clear how the limit function $u$ is determined qualitatively, 
that is, it is not clear which equation $u$ satisfies. 
In the homogenization theory, typically, the limit equation is given by 
the effective Hamiltonian, which is determined through the cell problem. 
Classically, we first establish the existence of solutions to the cell problem by considering the vanishing discount problem, and the effective Hamiltonian is uniquely determined. However, in our problem, this is not possible because of the lack of a priori Lipschitz estimate of the discount approximation. Therefore, we choose another approach. Roughly speaking, we determine the effective Hamiltonian by using the effective metric function. More precisely, let $m(t,x,y,c)$ be the function defined in Theorem \ref{thm:fundamental-sol}. Due to Corollary \ref{cor:m-bar}, we have the effective metric function, that is, 
\[\overline{m}(t,x,y,c):=\lim_{\ep\to0}\ep m\Big(\frac{t}{\ep},\frac{x}{\ep},\frac{y}{\ep},\frac{c}{\ep}\Big)
\quad\text{for all} \ (t,x,y,c)\in[0,\infty)\times\R^n\times\R^n\times\R. 
\]
By using $\overline{m}$, we determine the effective Lagrangian $\overline{L}:\R^n\to\R$ and 
the corresponding effective Hamiltonian $\overline{H}:\R^n\to\R$. Moreover, based on the optimal convergence rate in Theorem \ref{thm1}, we obtain the existence of bounded continuous correctors solving the cell problem
\begin{equation}\label{cell}
v_\tau+H(y,p\cdot y+v-\overline{H}(p)\tau,p+D_y v)=\overline{H}(p),\quad y\in\R^n,\ \tau>0
\end{equation}
introduced in \cite{IM}. 
Our result is stronger than \cite[Theorem 3]{IM}, where they were only able to obtain semi-continuous subcorrectors and supercorrectors. Here we note that the cell problem \eqref{cell} is defined for $(y,\tau)\in \R^n\times(0,\infty)$ instead of $y\in \T^n$, which makes the situation more complicated.
\begin{thm}\label{thm:qualitative}
Assume {\rm(H1)--(H5)}.  
Then, $u=\lim_{\ep \to 0}u^\ep$ given in Theorem \ref{thm1} is of the form
\[u(x,t)=\inf_{y\in\R^n}\bigg\{\varphi(y)+t\overline{L}\Big(\frac{x-y}{t}\Big)\bigg\} \quad \text{ for } (x,t) \in \R^n \times [0,\infty),\]
where $\overline{L}:\R^n\to\R$ is the convex and superlinear function given by Lemma \ref{lem:form}. 
Define
\begin{equation}\label{def:eff-H}
\overline{H}(p)=\sup_{v\in\R^n}\{p\cdot v-\overline{L}(v)\},
\end{equation} then $\overline{H}$ is convex, superlinear, and locally Lipschitz continuous, and $u$ solves
\begin{equation}\label{e0}
\begin{cases}
u_t+\ol{H}(Du)=0 \quad &\text{for} \ x\in\R^n, \ t>0, \\ 
u(x,0)=\varphi(x) \quad &\text{for} \ x\in\R^n. 
\end{cases}
\end{equation}
Moreover, for any $p\in\R^n$ and $w_0\in \BUC(\R^n)\cap\Lip(\R^n)$, let $w$ be the unique solution to 
\begin{equation}\label{Hp}
\begin{cases}
w_\tau+H(y,p\cdot y+w,p+D_y w)=0 \quad&\text{for} \ y\in\R^n,\ \tau>0,
\\ w(y,0)=w_0(y) \quad&\text{for} \ y\in\R^n, 
\end{cases}
\end{equation}
in $\BUC(\R^n\times[0,\infty))$. 
Then, there is a constant $C$ depending only on $H$, $\|w_0\|_{L^\infty(\R^n)}$, and $p$ such that 
\[
|w(y,\tau)+\overline{H}(p)\tau|\leq C
\quad \text{for\ all} \ y\in\R^n, \ \tau>0.  
\]
\end{thm}
 As explained above, we have not been able to obtain the global Lipschitz estimate for the solution of \eqref{Hp}. 
 It is thus surprising that we can control the oscillation of $w$ to get bounded continuous correctors by using Theorem \ref{thm1}. The requirement $w_0\in\Lip(\R^n)$ is to ensure the local Lipschitz continuity of $w$. Then by Lemma \ref{lem:CP}, $w$ is the unique solution of \eqref{Hp} in $\BUC(\R^n\times[0,\infty))$. 
 In \cite{IM}, the authors used the sliding method to improve the periodicity, then finally got the existence of bounded sub/super-correctors, which is quite delicate. It is worth mentioning that \cite[Theorem 1]{IM} only considered the large time average of \eqref{Hp} with initial data $w(y,0)=0$. 
 In \cite[Theorem 1.9 (2)]{NWY}, the authors proved a result similar to the existence of bounded correctors. 
 The argument there is inspired by the rotation number of the homeomorphisms of the circle. 
 Since we cannot control the oscillation of $w$ in $y$ without using Theorem \ref{thm1}, it is hard to apply the argument in \cite[Theorem 1.9 (2)]{NWY} for irrational vector $p$.
 
 As shown in Theorem \ref{thm:qualitative}, the large time average of \eqref{Hp} can be characterized by the effective Hamiltonian $\overline{H}(p)$. 
 However, the large time behavior could be very complicated as \eqref{Hp} may have non-trivial time periodic solutions. We refer the reader to \cite{NW,RWY,WYZ} for related works on the large time behavior, where the Hamilton-Jacobi equations depend non-monotonously on the unknown function.

Based on Theorem \ref{thm:qualitative}, we provide two representation formulas for the effective Hamiltonian in Lemma \ref{lem:rep-H-bar-1} and Proposition \ref{prop:inf-sup}.
In particular, Proposition \ref{prop:inf-sup} gives a weak analog of the inf-sup formula.
See Section \ref{ssec:rep-formula} for further discussions.

\smallskip

 Next, we investigate in more depth an improved regularity of $u^\ep$, the solution of \eqref{e}, as well as the correctors. As explained above, we can only get the Lipschitz estimate of the form $|u^\ep_t(x,t)|\leq e^{Kt/\ep}M$ for a.e. $(x,t)\in \R^n\times [0,\infty)$, which is not uniform in $\ep\in(0,1)$ and in $t>0$. 
 In the following theorem, we provide a global H\"older estimate for $u^\ep$, which is uniform in both $\ep$ and $t$. 
 Our result is inspired by \cite{CC,CS}. 
 We make a further assumption on $H$:
 \begin{itemize}
\item[(H6)] there are constants $\alpha_0,\beta_0>0$, $q_2\geq q_1>1$ such that
\begin{equation*}
\begin{cases}
    |H(y,r,p)-H(0,0,p)|\leq \alpha_0\quad &\text{for\ all}\quad (y,r,p)\in\R^n\times\R\times\R^n,
    \\  H(0,0,p)\leq \beta_0(|p|^{q_2}+1)\quad &\text{for\ all}\quad p\in\R^n,
    \\  |D_p H(0,0,p)|\geq \frac{1}{\beta_0}|p|^{q_1-1}-\beta_0\quad &\text{for a.e.}\quad p\in\R^n.
\end{cases}
\end{equation*}
 \end{itemize}
We note that the last inequality in (H6) implies 
\[
H(0,0,p)\geq \frac{1}{\beta_0}|p|^{q_1}-\beta_0 \quad \text{ for } p\in \R^n,
\]
where we take a larger $\beta_0$ if necessary.

\begin{thm}\label{thm:holder}
Assume {\rm (H1)--(H6)} and $\varphi\in\BUC(\R^n)$. There exists $C>0$ depending on $H$ and $\|\varphi\|_{W^{1,\infty}(\R^n)}$ such that for all $x,y\in\R^n$, $t,s>1$, and for all $\ep\in(0,1)$,
\begin{equation}\label{eq:u-ep-Holder}
|u^\ep(x,t)-u^\ep(y,s)|\leq C\Big(|x-y|^{\frac{m_1}{m_1+m_2-1}}+|t-s|^{\frac{m_1}{m_1+m_2}}\Big).
\end{equation}
\end{thm}

 Different from \cite{CC,CS}, we do not need to require the powers $q_1$ and $q_2$ in (H6) to be equal, and our estimate \eqref{eq:u-ep-Holder} is global with explicit H\"older exponents.
 To get the estimate of $u^\ep$ uniform in $\ep\in (0,1)$ and $t,s>1$, we use crucially the convergence rate in Theorem \ref{thm1}.

\smallskip

Using a similar argument, we can establish a global H\"older estimate for the solution $w$ of \eqref{Hp} for all $\tau>1$. Since $w_0\in\Lip(\R^n)$, for $\tau\in(0,2)$, we can obtain a Lipschitz estimate for $w(y,\tau)$, which depends only on $H$, $p$, and $\|Dw_0\|_{L^{\infty}(\R^n)}$. Combining two parts $\tau\in(0,2)$ and $\tau>1$, we can obtain the global regularity of $w(y,\tau)$ for all $\tau>0$.

\begin{thm}\label{thm:corrector-Holder}
Assume {\rm(H1)--(H6)}. 
Fix $w_0\in\BUC(\R^n)\cap\Lip(\R^n)$.
Then, the unique solution $w(y,\tau)$ of \eqref{Hp} in $\BUC(\R^n\times[0,\infty))$ satisfies
\[
|w(y_1,\tau_1)-w(y_2,\tau_2)|\leq C\Big(|y_1-y_2|^{\frac{m_1}{m_1+m_2-1}}+|\tau_1-\tau_2|^{\frac{m_1}{m_1+m_2}}\Big),
\]
for all $y_1,y_2\in\R^n$, $\tau_1,\tau_2>0$, with the constant $C$ depending only on $H$, $p$, and $\|w_0\|_{L^{\infty}(\R^n)}$. 

As a consequence, there exists a globally bounded, H\"older continuous corrector $v=v(y,\tau)$ satisfying 
\begin{equation*}
v_\tau+H(y,p\cdot y+v-\overline{H}(p)\tau,p+D_y v)=\overline{H}(p)\quad \text{ for } (y,\tau) \in\R^n\times \R.
\end{equation*}

\end{thm}

Theorems \ref{thm:holder}, \ref{thm:corrector-Holder} are particularly interesting and intrinsic, as quantitative homogenization results can provide uniform estimates and improved regularity for solutions to the original problems. 

\bigskip
\noindent
\textbf{Organization of the paper. } 
 In Section \ref{sec:vari}, we establish the implicit variational principle and the fundamental solution associated with \eqref{e}. 
Section \ref{sec:proof} is devoted to the proof of the subadditivity and superadditivity of the fundamental solution, and the proof of Theorem \ref{thm1}.
The proofs of Theorem \ref{thm:qualitative} and the representation formulas of the effective Hamiltonian (Lemma \ref{lem:rep-H-bar-1} and Proposition \ref{prop:inf-sup}) are given in Section \ref{sec:effective}. 
A De Giorgi's conjecture and a possible generalization of the Aubry--Mather theory are discussed in Section \ref{further}.
In Section \ref{sec:Holder}, we prove Theorem \ref{thm:holder}--\ref{thm:corrector-Holder}.
In Appendix \ref{sec:conv}, we obtain the comparison principle (Lemma \ref{lem:CP}) for \eqref{e} with $\ep=1$. 
Finally, in Appendix \ref{sec:moti}, we present a derivation of \eqref{e} from the perspective of dislocation dynamics.

\section{Fundamental solutions and implicit variational problem}\label{sec:vari}

In this section, we first establish a fundamental solution associated with 
the equation $u_t+H(x,u,Du)=0$ under assumptions (H1)--(H4). 
We define the Lagrangian $L:\R^n\times\R\times\R^n\to\R$ by 
\[
L(y,r,v):=\sup_{p\in\mathbb R^n}\Big\{v\cdot p-H(y,r,p)\Big\} 
\quad\text{for all} \ (y,r,v)\in\R^n\times\R\times\R^n. 
\]
Then, $L$ satisfies 
\begin{itemize}
\item[(L1)] 
$L\in C(\R^n\times\R\times\R^n)$, 
$y\mapsto L(y,r,v)$ is $\mathbb Z^n$-periodic for all $(r,v)\in\R\times\R^n$; 
\item[(L2)] 
$r\mapsto L(y,r,v)$ is Lipschitz continuous with a Lipschitz constant $K$ which is given by (H2), and $\Z$-periodic for $(y,v)\in\R^n\times\R^n$;
\item[(L3)]  
$L$ is superlinear in $v$, i.e., 
\[
\lim_{|v|\to\infty}\left(\inf_{(y,r)\in \R^n\times \R}\frac{L(y,r,v)}{|v|}\right)=\infty; 
\]
\item[(L4)] $v\mapsto L(y,r,v)$ is convex for all $(y,r)\in\R^n\times\R$.
\end{itemize}

\begin{thm}\label{thm:fundamental-sol}
Let $T>0$, $y\in\R^n$, and $c\in\R$. 
There exists a unique continuous function $(x,t)\mapsto m(t,x,y,c)$ defined on $\R^n\times(0,T]$ satisfying
\begin{equation}\label{minf}
m(t,x,y,c)=c+\inf_{\gamma\in\cC(x,y;t,0)}\int_0^tL(\gamma(s),m(s,\gamma(s),y,c),\dot{\gamma}(s))\, ds
\quad\text{for all} \ (x,t)\in\R^n\times(0,T], 
\end{equation}
where we set 
\[
\cC(x,y;b,a):=
\big\{\gamma\in\AC\big([a,b];\R^n\big)\mid \gamma(a)=y,
\gamma(b)=x\big\}
\]
for $a,b\in\R^n$ with $a\le b$ and $x,y\in\R^n$. 
Moreover, there exists a minimizer $\gamma\in\cC(x,y;t,0)$ 
of $m(t,x,y,c)$, that is, 
\[
m(t,x,y,c)=c+\int_0^tL(\gamma(s),m(s,\gamma(s),y,c),\dot{\gamma}(s))\, ds. 
\]
\end{thm}
\begin{rem}
For later purposes, for each $x,y\in\R^n$ and $c\in\R$, we extend the function $m$ to the function on $[0,T]\times\R^n\times\R^n\times\R$ by setting 
\begin{equation}\label{mt0}
m(0,x,y,c):=
\begin{cases}
c\quad &\text{if} \quad x=y,\\  
\infty
\quad &\text{if} \quad  x\neq y.
\end{cases}
\end{equation}
\end{rem}

The implicit variational problem of the type \eqref{minf} was studied in 
\cite{IWWY, NWY,WWY,WWY2} with several different settings. 
Since the setting of our paper is slightly different, we give the proof of Theorem \ref{thm:fundamental-sol} for completeness. 
The function $m$ given by Theorem \ref{thm:fundamental-sol} is sometimes called a \textit{fundamental solution} associated with $u_t+H(y,u,Du)=0$ in $\R^n\times(0,\infty)$,  
which is a generalization of metric functions of the convex Hamilton--Jacobi equation $u_t+H(x,Du)=0$ in $\R^n\times(0,\infty)$ 
(see \cite{Hung-book, TY} for instance).

\begin{proof}
Let $(y,c)\in\R^n\times\R$. 
We divide the proof into several steps.

\smallskip

\noindent \textbf{Step 1}. 
For $\phi\in C(\R^n\times(0,T])$, we define the operator $\mathcal A$ by
\[
\mathcal A\phi(x,t)
:=c+\inf_{\gamma\in\cC(x,y;t,0)}
\int_0^tL(\gamma(s),\phi(\gamma(s),s),\dot{\gamma}(s))\, ds.  
\]
By the standard theory of the calculus of variations, there is a minimizer $\gamma:[0,t]\to\R^n$ for $\mathcal A \phi(x,t)$ and each $(x,t)\in\R^n\times(0,T]$. 
We first prove that $(x,t)\mapsto \mathcal A\phi(x,t)$ is continuous on $\R^n\times(0,T]$. 

\smallskip
First, note that 
\begin{equation}\label{LK}
|L(y,r_1,v)-L(y,r_2,v)|\leq K
\quad\text{for all} \ y\in\R^n, r_1,r_2\in\R, v\in\R^n,
\end{equation}
by (L1), (L2). 
Take small $\delta>0$, and fix 
$0<t',t<T$, $x',x\in\R^n$ so that $|t-t'|<\delta$.
Then, we have 
\begin{align*}
&\mathcal A\phi(x,t)\\
\geq \ &c+\int_0^{t-\delta}L(\gamma(s),\phi(\gamma(s),s),\dot\gamma(s))\, ds+\inf_{\alpha\in\cC(x,\gamma(t-\delta);\delta,0)}
\int_0^\delta L(\alpha(s),0,\dot\alpha(s))\, ds-K\delta
\\ 
=\ &c+\int_0^{t-\delta}L(\gamma(s),\phi(\gamma(s),s),\dot\gamma(s))\, ds+m^0(\delta,x,\gamma(t-\delta))-K\delta,
\end{align*}
where $m^0:(0,T]\times\R^n\times\R^n\to\R$ is defined by
\begin{equation}\label{m0}
m^0(t,x,y)
:=
\inf_{\gamma\in\cC(x,y;t,0)}
\int_0^tL(\gamma(s),0,\dot\gamma(s))\, ds 
\quad\text{for} \ (t,x,y)\in(0,T]\times\R^n\times\R^n.
\end{equation}

On the other hand, 
for any $\alpha\in \cC(x',\gamma(t-\delta); t'-t+\delta,0)$, 
we set 
\[
\tilde{\gamma}(s)
:=
\begin{cases}
\gamma(s) & \text{for} \ s\in[0,t-\delta],\\  
\alpha(s-(t-\delta))
  & \text{for} \ s\in[t-\delta, t']. 
\end{cases}
\]
Then, $\tilde{\gamma}\in\cC(x', y;t',0)$, which implies 
\begin{align*}
\mathcal A\phi(x',t')
&\leq 
c+\int_0^{t'}L(\tilde{\gamma}(s),\phi(\tilde{\gamma}(s),s),\dot{\tilde{\gamma}}(s))\, ds\\
&=\, 
c+\int_0^{t-\delta}L(\gamma(s),\phi(\gamma(s),s),\dot\gamma(s))\, ds
+\int_0^{t'-t+\delta} L(\alpha(s),\phi(\gamma(s),s),\dot\alpha(s))\, ds. 
\end{align*}
Taking the infimum on $\alpha\in\cC(x',\gamma(t-\delta); t'-t+\delta,0)$, 
we obtain 
\[
\mathcal A\phi(x',t')
\le 
c+
\int_0^{t-\delta}L(\gamma(s),\phi(\gamma(s),s),\dot\gamma(s))\, ds
+m^0(t'-t+\delta,x',\gamma(t-\delta))+K(t'-t+\delta). 
\]
Then, we have
\[
\mathcal A\phi(x',t')-\mathcal A\phi(x,t)\leq m^0(t'-t+\delta,x',\gamma(t-\delta))-m^0(\delta,x,\gamma(t-\delta))+K(t'-t+2\delta).
\]

Since $\gamma$ is continuous, we know that $r_0:=\sup_{s\in[0,t]}|\gamma(s)|<\infty$. By \cite[Theorem 3.1]{D}, $m^0(t,x,y)$ is locally Lipschitz continuous on $(0,\infty)\times\R^n\times\R^n$. Let $\ell$ be a Lipschitz constant of $(s,z)\mapsto m^0(s,z,y)$ on $[\frac{\delta}{2},T]\times B_{r_0+\frac{\delta}{2}}(0)\times B_{r_0}(0)$. If $0<\theta\leq \frac{\delta}{2}$ and $|t'-t|+|x'-x|\leq \theta$, we have $t'-t+\delta\in [\frac{\delta}{2},T]$ and $|x'|\leq |x|+|x'-x|\leq r_0+\frac{\delta}{2}$. 
Then, 
\[|m^0(t'-t+\delta,x',\gamma(t-\delta))-m^0(\delta,x,\gamma(t-\delta))|\leq \ell \theta.\]
For each $\epsilon>0$, taking $\delta=\frac{\epsilon}{4K}$ and $\theta=\min\{\frac{\delta}{2},\frac{\epsilon}{4\ell}\}$, we obtain 
\[
\mathcal A\phi(t',x')-\mathcal A\phi(t,x)
\leq 
\ell \theta+K|t'-t|+2K\delta<\epsilon.  
\]
By symmetry, we can prove the reverse inequality.

\medskip
\noindent \textbf{Step 2}. 
We next find a fixed point of $\cA$. 
Let $\psi_1,\psi_2\in C(\R^n\times(0,T])$ with \[\sup_{(z,s)\in\R^n\times(0,T]}|\psi_2(z,s)-\psi_1(z,s)|<\infty.\] Let $\xi_1$ be a minimizer of $\mathcal A\psi_1(x,t)$. We have
\begin{align*}
&\mathcal A\psi_2(x,t)-\mathcal A\psi_1(x,t)
\\ &\leq \int_0^t L(\xi_1(s),\psi_2(\xi_1(s),s),\dot\xi_1(s))-L(\xi_1(s),\psi_1(\xi_1(s),s),\dot\xi_1(s))\, ds
\\ &\leq Kt\sup_{(z,s)\in \R^n\times(0,T]}|\psi_2(z,s)-\psi_1(z,s)|.
\end{align*}
By symmetry, we get
\begin{equation}\label{psi}
|\mathcal A\psi_2(x,t)-\cA\psi_1(x,t)|\leq Kt\sup_{(z,s)\in \R^n\times(0,T]}|\psi_2(z,s)-\psi_1(z,s)|.
\end{equation}
We take $\phi_0\equiv 0$ and define the family of functions 
by setting $\phi_n:=\mathcal A^n\phi_0$ for all $n\in\N$. 
Let $\gamma_1$ be a minimizer of $\phi_1(x,t)$. We have
\begin{align*}
&\phi_2(x,t)-\phi_1(x,t)
\\ 
\leq&\,
 \int_0^t L(\gamma_1(s),\phi_1(\gamma_1(s),s),\dot\gamma_1(s))-L(\gamma_1(s),0,\dot\gamma_1(s))\, ds\leq Kt.
\end{align*}
By symmetry, we get $|\phi_2(x,t)-\phi_1(x,t)|\leq Kt$. 

Next, we prove
\[
|\phi_{n+1}(x,t)-\phi_n(x,t)|\leq \frac{K^n}{n!}t^n
\quad\text{for all} \ n\in\N
\]
by induction. 
Assume that the above inequality holds for $n=k-1$ with $k\ge2$. 
Then, let $\gamma_{k}$ be a minimizer of $\phi_k(x,t)$. We have
\begin{align*}
&\phi_{k+1}(x,t)-\phi_k(x,t)
\\ &\leq \int_0^t L(\gamma_k(s),\phi_k(\gamma_k(s),s),\dot\gamma_k(s))-L(\gamma_k(s),\phi_{k-1}(\gamma_k(s),s),\dot\gamma_k(s)) \, ds
\\ &\leq K\int_0^t|\phi_k(\gamma_k(s),s)-\phi_{k-1}(\gamma_k(s),s)|\, ds \leq K\int_0^t\frac{K^{k-1}}{(k-1)!}s^{k-1}\, ds=\frac{K^k}{k!}t^k.
\end{align*}
By symmetry, we get $|\phi_{k+1}(x,t)-\phi_k(x,t)|\leq \frac{K^k}{k!}t^k$. Then,
\[\phi_n(x,t)=\sum_{i=0}^{n-1}(\phi_{i+1}(x,t)-\phi_i(x,t))\]
converges on $\R^n\times(0,T]$ uniformly to $\phi_\infty(x,t)\in C(\R^n\times(0,T])$ as $n\to\infty$. By \eqref{psi}, we have
\begin{align*}
|\mathcal A\phi_\infty(x,t)-\phi_\infty(x,t)|
&
\leq 
|\mathcal A\phi_\infty(x,t)-\phi_{n+1}(x,t)|+|\phi_{n+1}(x,t)-\phi_\infty(x,t)|\\
&=
|\mathcal A\phi_\infty(x,t)-\mathcal A\phi_{n}(x,t)|+|\phi_{n+1}(x,t)-\phi_\infty(x,t)|
\\ 
&\leq Kt\sup_{(z,s)\in \R^n\times(0,T]}|\phi_\infty(z,s)-\phi_n(z,s)|+|\phi_{n+1}(x,t)-\phi_\infty(x,t)|.
\end{align*}
Letting $n\to\infty$, we get $\mathcal A\phi_\infty=\phi_\infty$, that is, $\phi_\infty$ satisfies \eqref{minf}.

\medskip
\noindent \textbf{Step 3.} 
We finally prove the uniqueness of $m$ which satisfies \eqref{minf}. 
Assume that there are two functions $m_1, m_2\in C(\R^n\times(0,T])$ satisfying
\[m_i(x,t)=c+\inf_{\gamma\in\cC(x,y;t,0)}\int_0^tL(\gamma(s),m_i(\gamma(s),s),\dot{\gamma}(s))\, ds,\quad i=1,2,\]
and also there is $(x,t)\in\R^n\times(0,T]$ such that $m_1(x,t)>m_2(x,t)$. Let $\eta:[0,t]\to\R^n$ be a minimizer of $m_2(x,t)$. 
Then, $s\mapsto L(\eta(s),m_2(\eta(s),s),\dot\eta(s))\in L^1([0,t])$, which implies that $s\mapsto L(\eta(s),m_1(\eta(s),s),\dot\eta(s))\in L^1([0,t])$ by \eqref{LK}. 
Thus, by the absolute continuity of the Lebesgue integral, $m_1(\eta(s),s)\to c$ and $m_2(\eta(s),s)\to c$ as $s\to 0$.  
Define 
\[
F(s):=m_1(\eta(s),s)-m_2(\eta(s),s)
\quad \text{for all} \ s\in[0,t]. 
\]
Since $F$ is continuous on $(0,t]$, $\lim_{s\to 0}F(s)=0$, and $F(t)>0$. By continuity, there is $\sigma\in[0,t)$ such that $F(\sigma)=0$ and $F(\tau)>0$ for all $\tau\in(\sigma,t]$.

Fix any $s\in(\sigma,t]$. Note that 
\[
m_1(\eta(s),s)
=c+\inf_{\eta\in\cC(\eta(s),y;s,0)}
\int_0^{s}L(\eta(\tau),m_1(\eta(\tau),\tau),\dot{\eta}(\tau))\,d\tau. 
\]
Take any $\alpha\in\cC(\eta(\sigma),y;\sigma,0)$, and set 
\[
\gamma(\tau):=
\left\{
\begin{array}{ll}
\alpha(\tau) & \text{for} \ \tau\in[0,\sigma], \\
\eta(\tau) & \text{for} \ \tau\in[\sigma,s]. 
\end{array}
\right. 
\]
Then, 
$\gamma\in\cC(\eta(s),y;s,0)$, which implies 
\begin{align*}
m_1(\eta(s),s)\le&\, 
c+\int_0^{s}L(\gamma(\tau),m_1(\gamma(\tau),\tau),\dot{\gamma}(\tau))\,d\tau\\
=&\, 
c
+\int_0^{\sigma}L(\alpha(\tau),m_1(\alpha(\tau),\tau),\dot{\gamma}(\alpha))\,d\tau
+\int_{\sigma}^{s}L(\eta(\tau),m_1(\eta(\tau),\tau),\dot{\eta}(\alpha))\,d\tau. 
\end{align*}
Taking the infimum on $\alpha\in\cC(\eta(\sigma),y;\sigma,0)$ yields 
\begin{align*}
m_1(\eta(s),s)\le&\, 
m_1(\eta(\sigma),\sigma)
+\int_{\sigma}^{s}L(\eta(\tau),m_1(\eta(\tau),\tau),\dot{\eta}(\tau))\,d\tau. 
\end{align*}
Noting that $\eta$ is a minimizer of $m_2(x,t)$, we have 
\[
m_2(\eta(s),s)
=
m_2(\eta(\sigma),\sigma)
+\int_{\sigma}^sL(\eta(\tau),m_2(\eta(\tau),\tau),\dot\eta(\tau))\, d\tau,
\]
which, together with $F(\sigma)=0$, implies that
\[F(s)\leq
\int_\sigma^s
L(\eta(\tau),m_1(\eta(\tau),\tau),\dot{\eta}(\tau))
-
L(\eta(\tau),m_2(\eta(\tau),\tau),\dot{\eta}(\tau))
\,d\tau
\le 
K\int_\sigma^sF(\tau)\, d\tau.\]
By the Gronwall inequality, we have $F(s)=0$ for all $s\in(\sigma,t]$ as $F(\sigma)=0$, which contradicts $F(t)>0$.
\end{proof}

Next, we assume (H5) and establish a representation formula for the solution to \eqref{e} with $\ep=1$, that is, the Cauchy problem 
\begin{numcases}
{}
u_t+H(x,u, Du)=0 \quad & \text{for} \ $x\in\R^n, t>0$, \label{eq:ini-1}\\
u(x,0)=\varphi(x) \quad & \text{for} \ $x\in\R^n$. \label{eq:ini-2}
\end{numcases}

\begin{lem}\label{lem:impli-rep-form}
There exists a continuous solution $u$ to \eqref{eq:ini-1}--\eqref{eq:ini-2} satisfying 
\begin{equation}\label{uinf}
u(x,t)=\inf_{\gamma\in\cC(x;t)}
\bigg\{\varphi(\gamma(0))+\int_0^t
L(\gamma(s),u(\gamma(s),s),\dot\gamma(s))\, ds\bigg\} \quad \text{for $(x,t)\in \R^n\times [0,\infty)$.}
\end{equation}
Here, 
\[
\cC(x;t):=\{\gam\in\AC([0,t];\R^n)\mid \gam(t)=x\}. 
\]
For $(x,t)\in \R^n\times [0,\infty)$, the infimum in \eqref{uinf} can be achieved. 
Let $\gamma:[0,t]\to \R^n$ be a minimizer of $u(x,t)$, then there is a constant $M_0>0$ depending on $H$ and $\|D\varphi\|_{L^\infty(\R^n)}$ such that 
\[
|x-\gamma(0)|\leq M_0t.
\]
\end{lem}
\begin{proof}
We first show there exists a continuous function $u$ satisfying \eqref{uinf}. Define $u_0=0$ and for $n\in\mathbb N$, we consider the following sequence
\[u_{n+1}(x,t)=\inf_{\gamma\in\cC(x;t)}
\bigg\{\varphi(\gamma(0))+\int_0^t
L(\gamma(s),u_n(\gamma(s),s),\dot\gamma(s))\, ds\bigg\}.\]
By (L3), there is $K_1>0$ depending on $\|D\varphi\|_{L^\infty(\R^n)}$ such that 
\[
L(y,0,v)
\ge 
(1+\|D\varphi\|_{L^\infty(\R^n)})|v|-K_1\quad
\text{for all} \ y\in\R^n,\ v\in\R^n.
\]
The minimizers of $u_{n+1}(x,t)$ are contained in the set of curves $\gamma\in\cC(x;t)$ satisfying
\[u_{n+1}(x,t)+1\geq \varphi(\gamma(0))+\int_0^t
L(\gamma(s),u_n(\gamma(s),s),\dot\gamma(s))\, ds.\]
By taking the constant curve equaling $x$, we also have
\[u(x,t)+1\leq \varphi(x)+\Big(\sup_{y\in\R^n}L(y,0,0)+K\Big)t+1,\]
which implies
\begin{align*}
\varphi(x)+\Big(\sup_{y\in\R^n}L(y,0,0)+K\Big)t+1
&\geq \varphi(\gamma(0))+\int_0^t
L(\gamma(s),u_n(\gamma(s),s),\dot\gamma(s))\, ds
\\ &\geq \varphi(\gamma(0))+(1+\|D\varphi\|_{L^\infty(\R^n)})|x-\gamma(0)|-(K_1+K)t.
\end{align*}
Since $\varphi(x)-\varphi(\gamma(0))\leq \|D\varphi\|_{L^\infty(\R^n)}|x-\gamma(0)|$, we conclude
\[|x-\gamma(0)|\leq \Big(\sup_{y\in\R^n}L(y,0,0)+K_1+2K\Big)t+1.\]
Since $\gamma(0)$ is bounded, by the classical result of a one-dimensional variational principle, we know the existence of minimizers of $u_{n+1}(x,t)$. Then, as in the proof of Theorem \ref{thm:fundamental-sol}, Step 2, we can show that
\[
|u_{n+1}(t,x)-u_n(t,x)|\leq \frac{K^n}{n!}t^n
\quad\text{for all} \ n\in\N,
\]
and that $u_n$ converges to a continuous function $u(x,t)$ uniformly, and $u(x,t)$ satisfies \eqref{uinf}. 
By the optimal control formula of viscosity solutions of convex Hamilton--Jacobi equations, it is rather standard to prove that $u$ is a viscosity solution to \eqref{eq:ini-1}--\eqref{eq:ini-2}. Note that since $L$ satisfies (L3), (L4), there exists a minimizer of $u(x,t)$. Moreover, as argued above, there is a constant $M_0$ depending on $L$ and $\|D\varphi\|_{L^\infty(\R^n)}$ such that the minimizers $\gamma$ of $u(x,t)$ satisfies $|x-\gamma(0)|\leq M_0t$.
\end{proof}

The following lemma demonstrates the failure of non-expansiveness for the solution to \eqref{eq:ini-1}--\eqref{eq:ini-2} with respect to the initial data, which presents a major difficulty in this work.

\begin{lem}
Let $\varphi_1,\varphi_2\in\Lip(\R^n)$.
For $i=1,2$, let $u_i$ satisfy the formula \eqref{uinf} with $\varphi_i$ in place of $\varphi$. 
If $\|\varphi_1-\varphi_2\|_{L^\infty(\R^n)}<\infty$, then we have
\begin{equation}\label{exp}
\|(u_1-u_2)(\cdot,t)\|_{L^\infty(\R^n)}\leq  \|\varphi_1-\varphi_2\|_{L^\infty(\R^n)}e^{Kt} \quad \text{ for all } t\geq 0.
\end{equation}
\end{lem}
\begin{proof}
Fix $(x,t)\in \R^n\times [0,\infty)$.
We assume $u_1(x,t)>u_2(x,t)$. 
Let $\gamma$ be a minimizer of $u_2(x,t)$. If there is $\sigma\in[0,t)$ such that $u_1(\gamma(\sigma),\sigma)=u_2(\gamma(\sigma),\sigma)$, then for all $s\in(\sigma,t]$, 
\[u_1(\gamma(s),s)\leq u_1(\gamma(\sigma),\sigma)+\int_\sigma^sL(\gamma(\tau),u_1(\gamma(\tau),\tau),\dot{\gamma}(\tau))\, d\tau,\]
and
\[u_2(\gamma(s),s)=u_2(\gamma(\sigma),\sigma)+\int_\sigma^sL(\gamma(\tau),u_2(\gamma(\tau),\tau),\dot{\gamma}(\tau))\,  d\tau.\]
Define
\[F(s):=u_1(\gamma(s),s)-u_2(\gamma(s),s).\]
Without loss of generality, we can assume $F(s)>0$ for $s\in (\sigma,t]$.
Since $F(\sigma)=0$, we have
\[F(s)\leq K\int_\sigma^sF(\tau)\, d\tau.\]
By the Gronwall inequality, we have $F(s)=0$ for all $s\in(\sigma,t]$, which leads to a contradiction. Thus, $u_1(\gamma(s),s)>u_2(\gamma(s),s)$ for all $s\in[0,t]$. For all $s\in[0,t]$, 
\[u_1(\gamma(s),s)\leq \varphi_1(\gamma(0))+\int_0^sL(\gamma(\tau),u_1(\gamma(\tau),\tau),\dot{\gamma}(\tau))\, d\tau,\]
and
\[u_2(\gamma(s),s)=\varphi_2(\gamma(0))+\int_0^sL(\gamma(\tau),u_2(\gamma(\tau),\tau),\dot{\gamma}(\tau))\,  d\tau.\]
Then,
\[F(s)\leq \varphi_1(\gamma(0))-\varphi_2(\gamma(0))+K\int_0^sF(\tau)\, d\tau.\]
By the Gronwall inequality, we have
\[u_1(\gamma(s),s)-u_2(\gamma(s),s)\leq \|\varphi_1-\varphi_2\|_{L^\infty(\R^n)}e^{Ks}.\]
Setting $s=t$, and by symmetry, we get the conclusion.
\end{proof}

\begin{lem}\label{lem:u-optimal-control}
Let $u$ be the unique solution of \eqref{eq:ini-1}--\eqref{eq:ini-2} in $\AL(\R^n\times[0,\infty))$, and then $u$ satisfies \eqref{uinf}, 
and furthermore $u\in \Lip(\R^n\times [0,T])$ for every $T>0$.
\end{lem}
\begin{proof}
Let $\gamma$ be a minimizer of $u(x,t)$ with $\gam(t)=x$ given by Lemma \ref{lem:impli-rep-form}. 
Then, we see that 
\begin{align*}
&\varphi(x)+\Big(\sup_{y\in\R^n}L(y,0,0)+K\Big)t
\geq u(x,t)\geq \varphi(\gamma(0))+\Big(\inf_{y\in\R^n,\,v\in\R^n}L(y,0,v)-K\Big)t.
\end{align*}
Moreover, as $|x-\gamma(0)|\leq M_0t$, we have
$
|\varphi(\gamma(0))-\varphi(x)|
\le\|D\varphi\|_{L^\infty(\R^n)}|x-\gamma(0)|
\le \|D\varphi\|_{L^\infty(\R^n)}M_0t, 
$
which implies 
\begin{align*}
u(x,t)&\le\, 
\varphi(x)+\Big(\sup_{y\in\R^n}L(y,0,0)+K\Big)t\\
&\le\, 
\|D\varphi\|_{L^\infty(\R^n)}M_0t
+|\varphi(\gamma(0))|
+\Big(\sup_{y\in\R^n}L(y,0,0)+K\Big)t\\
&\le\, 
\|D\varphi\|_{L^\infty(\R^n)}(2M_0t+|x|)
+|\varphi(0)|
+\Big(\sup_{y\in\R^n}L(y,0,0)+K\Big)t,  
\end{align*}
and 
\[
u(x,t)\ge 
\varphi(x)-\|D\varphi\|_{L^\infty(\R^n)}M_0t
+\Big(\inf_{y\in\R^n,\,v\in\R^n}L(y,0,v)-K\Big)t, 
\]
which implies that $u(x,t)$ has at most linear growth in $x$.

Let 
\[M:=\sup_{y\in\R^n,\, r\in\R,\, |p|\leq \|D\varphi\|_{L^\infty(\R^n)}}|H(y,r,p)|<\infty.\] 
One can check that $\varphi(x)\pm Mt$ are a subsolution and a supersolution to \eqref{eq:ini-1}--\eqref{eq:ini-2}, respectively. By the comparison principle, Lemma \ref{lem:CP}, we have
\[
\varphi(x)-Mt\leq u(x,t) \leq \varphi(x)+Mt \quad \text{ for all } (x,t)\in \R^n\times [0,\infty).
\]
In particular,
\[\|u(\cdot,s)-\varphi\|_{L^\infty(\R^n)}=\|u(\cdot,s)-u(\cdot,0)\|_{L^\infty(\R^n)}\leq Ms\quad \text{ for } s>0.\]
Then by \eqref{exp},
\begin{equation}\label{eks}
\|u(\cdot,t+s)-u(\cdot,t)\|_{L^\infty(\R^n)}\leq e^{Kt}\|u(\cdot,s)-u(\cdot,0)\|_{L^\infty(\R^n)}\leq e^{Kt}Ms\quad \text{for } t,s>0.
\end{equation}
Hence, for each $T>0$,
\[
\|u_t\|_{L^\infty(\R^n\times [0,T])}\leq e^{KT}M.
\]
By (H3),
\[
\|Du\|_{L^\infty(\R^n\times [0,T])}\leq C=C(T,H).
\]
Thus, $u\in \Lip(\R^n\times [0,T])$ for every $T>0$. 
By Lemma \ref{lem:CP} again, we know that $u$ is the unique solution to \eqref{eq:ini-1}--\eqref{eq:ini-2} in $\AL(\R^n\times[0,\infty))$.
\end{proof}

\begin{thm}\label{thm:rep} 
There exists a unique solution $u\in \AL(\R^n\times[0,\infty))$ to \eqref{eq:ini-1}--\eqref{eq:ini-2} represented by 
\begin{equation}\label{eq:u-m}
u(x,t)=\inf_{y\in\R^n}m(t,x,y,\varphi(y))
\quad\text{for all} \ (x,t)\in\R^n\times[0,\infty), 
\end{equation}
where $m$ is the fundamental solution to \eqref{eq:ini-1} given by 
Theorem {\rm\ref{thm:fundamental-sol}}. 
Moreover, the infimum is achieved at $y=\gamma(0)$, where 
$\gamma\in\cC(x;t)$ is a minimizer of \eqref{uinf}. 
\end{thm}
\begin{proof}
By Lemma \ref{lem:u-optimal-control}, the value function $u$ given by \eqref{uinf} is the unique solution of \eqref{eq:ini-1}--\eqref{eq:ini-2} in $\AL(\R^n\times[0,\infty))$.
It remains to show that $u$ satisfies \eqref{eq:u-m}.

Fix $(x,t)\in \R^n\times [0,\infty)$.
We first consider the case where $t>0$. 
Let $\gamma:[0,t]\to\R^n$ be a minimizer of $u(x,t)$ in \eqref{uinf}. 
By the uniqueness of continuous functions satisfying \eqref{minf}, 
we have $u(x,t)=m(t,x,\gamma(0),\varphi(\gamma(0)))$. 
Here, by contradiction, assume that 
\[
\inf_{y\in\R^n}m(t,x,y,\varphi(y))<u(x,t). 
\]
Then, there is $\bar y\in\R^n$ such that
$m(t,x,\bar y,\varphi(\bar y))<u(x,t)$. 
Take a minimizer $\xi\in\cC(x,\bar{y};t,0)$ of $m(t,x,\bar y,\varphi(\bar y))$. 
It is clear to see $u(\xi(0),0)=u(\bar{y},0)=\varphi(\bar{y})$. 
Moreover, noting that $s\mapsto L(\xi(s),m(s,\xi(s),\bar y,\varphi(\bar y)),\dot \xi(s))\in L^1([0,t])$ and
\[m(t,\xi(t),\bar y,\varphi(\bar y))=\varphi(\bar y)+\int_0^tL(\xi(\tau),m(\tau,\xi(\tau),\bar y,\varphi(\bar y)),\dot \xi(\tau))\, d\tau,\]
we obtain 
$\lim_{t\to 0}m(t,\xi(t),\bar y,\varphi(\bar y))=\varphi(\bar y)$. 

Set 
\[
F(s):=u(\xi(s),s)-m(s,\xi(s),\bar y,\varphi(\bar y))\quad 
\text{for} \ s\in[0,t].
\]
We repeat a similar argument in Step 3 of the proof of Theorem \ref{thm:fundamental-sol}  
to obtain a contradiction. 
Note that $F$ is continuous on $(0,t]$, and $\lim_{s\to 0}F(s)=0$, $F(t)>0$. 
By continuity, there is $\sigma\in[0,t)$ such that $F(\tau)>0$ for all $\tau\in(\sigma,t]$. Note that $F(\sigma)=0$. Then, for $s\in(\sigma,t]$,
\[
u(\xi(s),s)\leq u(\xi(\sigma),\sigma)+\int_{\sigma}^sL(\xi(\tau),u(\xi(\tau),\tau),\dot\xi(\tau))\, d\tau,\]
and
\[m(s,\xi(s),\bar y,\varphi(\bar y))=m(\sigma,\xi(\sigma),\bar y,\varphi(\bar y))+\int_{\sigma}^sL(\xi(\tau),m(\tau,\xi(\tau),\bar y,\varphi(\bar y)),\dot\xi(\tau))\, d\tau.\]
Therefore, for $s\in(\sigma,t]$,
\[F(s)\leq K\int_\sigma^sF(\tau)\, d\tau.\]
By the Gronwall inequality, we have $F(s)=0$ for all $s\in(\sigma,t]$, which contradicts $F(t)>0$.

Finally, for the case $t=0$, by \eqref{mt0}, if $x\neq y$, 
then $m(0,x,y,\varphi(y))=\infty$ by the definition. 
Hence, 
\[u(x,0)=\varphi(x)=m(0,x,x,\varphi(x))=\inf_{y\in\R^n}m(0,x,y,\varphi(y)).
\]
\end{proof}

\smallskip
Let us now give the following simple lemma, which will be essentially used to 
obtain the subadditivity and superadditivity of the fundamental solution. 
This is related to the implicit variational principle, which is the so-called \textit{Herglotz variational principle} (see \cite{CCJWY}).

Given $\gamma:[0,t]\to\mathbb R^n$ with $L(\gamma(s),0,\dot\gamma(s))\in L^1([0,t])$, we consider the ordinary differential equation 
\begin{equation}\label{dotu}
\begin{cases}
\dot{\xi}(s)=L(\gamma(s),\xi(s),\dot\gamma(s))\quad 
\text{ for  a.e. } s\in[0,t],
\\ \xi(0)=c,
\end{cases}
\end{equation}
for $c\in\R$. 
In our setting, 
\eqref{dotu} admits a unique absolutely continuous solution, 
which is denoted by $\xi=\xi_\gamma$.  By \eqref{LK}, $L(\gamma(s),\xi_\gamma(s),\dot\gamma(s))\in L^1([0,t])$.

\begin{lem}\label{m<dotu}
Let $c\in\R$, $t>0$ and $\gamma\in\AC([0,t];\R^n)$ with $L(\gamma(s),0,\dot\gamma(s))\in L^1([0,t])$, 
and let $\xi_\gamma$ be the solution to \eqref{dotu}, and $m$ be the fundamental solution to \eqref{eq:ini-1}. 
Then, 
\[
m(t,\gamma(t),\gamma(0),c)\leq \xi_\gamma(t). 
\]
\end{lem}
\begin{proof}
By contradiction, assume that
\[m(t,\gamma(t),\gamma(0),c)>\xi_\gamma(t).
\]
We repeat a similar argument in Step 3 of the proof of Theorem \ref{thm:fundamental-sol} 
to obtain a contradiction. 
Set $F(s):=m(s,\gamma(s),\gamma(0),c)-\xi_\gamma(s)$ for $s\in[0,t]$. 
Since $F$ is continuous on $(0,t]$, $\lim_{s\to 0}F(s)=0$ and $F(t)>0$. 
By continuity, there is $\sigma\in[0,t)$ such that $F(\sigma)=0$ and $F(\tau)>0$ for all 
$\tau\in(\sigma,t]$. Note that 
\[
m(s,\gamma(s),\gamma(0),c)\leq m(\sigma,\gamma(\sigma),\gamma(0),c)+\int_{\sigma}^sL(\gamma(\tau),m(\tau,\gamma(\tau),\gamma(0),c),\dot\gamma(\tau))\, d\tau,
\]
and
\[
\xi_\gamma(s)=\xi_\gamma(\sigma)+\int_{\sigma}^sL(\gamma(\tau),\xi_\gamma(\tau),\dot\gamma(\tau))\, d\tau
\quad\text{for} \ s\in(\sigma,t], 
\]
which implies that 
\[
0\le F(s)\le K\int_{\sigma}^{s}F(\tau)\,d\tau.
\]
By the Gronwall inequality, we have $F(s)=0$ for all $s\in(\sigma,t]$, which contradicts $F(t)>0$. 
\end{proof}

\section{Proof of Theorem \ref{thm1}}\label{sec:proof}
In this section, we give a proof of Theorem \ref{thm1}.

\begin{prop}\label{prop:uep=m}
Let $u^\ep$ be the solution to \eqref{e}. 
We have 
\begin{equation}\label{eq:inf-m-ep}
u^\varepsilon(x,t)
=
\inf_{y\in\mathbb R^n} m^\ep(t,x,y,\varphi(y)) 
\quad
\text{for all} \ (t,x,y,c)\in [0,\infty)\times\R^n\times\R^n\times\R, 
\end{equation}
where 
\[
m^\ep(t,x,y,c):=\ep m\left(\frac{t}{\ep}, \frac{x}{\ep}, \frac{y}{\ep}, \frac{c}{\ep}\right) \quad \text{ for all $(t,x,y,c)\in [0,\infty)\times\R^n\times\R^n\times\R$,}
\]
and $m$ is the fundamental solution to \eqref{eq:ini-1}. 
Moreover, there is a constant $M_0>0$ depending on $H$ and 
$\|D\varphi\|_{L^\infty(\R^n)}$ such that any minimizing point $y$ of \eqref{eq:inf-m-ep} satisfies 
\[
|x-y|\le M_0t.
\]
\end{prop}
\begin{proof}
We first prove \eqref{eq:inf-m-ep}. 
This is an easy consequence of the change of variables, but we give the proof for the readability. 
Note that $u^\ep$ satisfies 
\begin{align*}
u^\varepsilon(x,t)
=&\, 
\inf_{\gamma\in\cC(x;t)}
\bigg\{\varphi(\gamma(0))+\int_0^tL\Big(\frac{\gamma(s)}{\varepsilon},\frac{u^\varepsilon(\gamma(s),s)}{\varepsilon},\dot{\gamma}(s)\Big)\, ds\bigg\}
 \\
=&\, 
\inf_{\eta\in\cC(\frac{x}{\ep};\frac{t}{\varepsilon})}\bigg\{\varphi(\varepsilon\eta(0))+\varepsilon\int_0^{\frac{t}{\varepsilon}}L\Big(\eta(\tau),\frac{u^\varepsilon(\varepsilon\eta(\tau),\varepsilon \tau)}{\varepsilon},\dot{\eta}(\tau)\Big)\, d\tau\bigg\}.  
\nonumber
\end{align*}
Setting $v^\ep(x,t):=\frac{u^\varepsilon(\varepsilon x,\varepsilon t)}{\varepsilon}$, 
we have 
\[v^\ep\Big(\frac{x}{\varepsilon},\frac{t}{\varepsilon}\Big)=\frac{u^\epsilon(x,t)}{\varepsilon}
=\inf_{\eta\in\mathcal{C}(\frac{x}{\varepsilon};\frac{t}{\varepsilon})}\bigg\{\frac{\varphi(\varepsilon\eta(0))}{\varepsilon}+\int_0^{\frac{t}{\varepsilon}}L\Big(\eta(\tau),v^\ep(\eta(\tau),\tau),\dot{\eta}(\tau)\Big)\, d\tau\bigg\}.\]
Therefore, 
\[
u^\varepsilon(x,t)=\varepsilon v^\ep\Big(\frac{x}{\varepsilon},\frac{t}{\varepsilon}\Big)
=
\varepsilon \inf_{y\in\mathbb R^n} m\Big(\frac{t}{\varepsilon},\frac{x}{\varepsilon},\frac{y}{\varepsilon},\frac{\varphi(y)}{\varepsilon}\Big)
=\inf_{y\in\mathbb R^n}m^\ep(t,x,y,\varphi(y)), 
\]
which completes the proof of \eqref{eq:inf-m-ep}.

We next prove the latter claim. 
In view of (L3), 
there is a constant $K_1>0$ depending on $\|D\varphi\|_{L^\infty(\R^n)}$ such that 
\[
L(y,0,v)
\ge 
(1+\|D\varphi\|_{L^\infty(\R^n)})|v|-K_1\quad
\text{for all} \ y\in\R^n,\ v\in\R^n.
\]
Then, by \eqref{LK}, for a minimizer $\gamma\in\cC(x;t)$ of $u^\ep(x,t)$, 
\begin{align*}
u^\ep(x,t)&\geq \varphi(\gamma(0))+\int_0^t L\Big(\frac{\gamma(s)}{\ep},0,\dot\gamma(s)\Big)\, ds-Kt
\\ &\ge \varphi(\gamma(0))+\int_0^t (1+\|D\varphi\|_{L^\infty(\R^n)})|\dot\gamma(s)|\, ds-K_1t-Kt
\\ &\ge \varphi(\gamma(0))+(1+\|D\varphi\|_{L^\infty(\R^n)})|x-\gamma(0)|-(K_1+K)t.
\end{align*}
On the other hand, we have 
\[
u^\ep(x,t)\le \varphi(x)+\int_0^tL\Big(\frac{x}{\ep},0,0\Big)\,ds+Kt\leq \varphi(x)+\Big(\max_{y\in\R^n}|L(y,0,0)|+K\Big)t.
\]
Thus, 
\begin{align*}
(1+\|D\varphi\|_{L^\infty(\R^n)})|x-\gamma(0)|&\le \varphi(x)-\varphi(\gamma(0))+\Big(\max_{y\in\R^n}|L(y,0,0)|+K_1+2K\Big)t
\\ &\le \|D\varphi\|_{L^\infty(\R^n)}|x-\gamma(0)|+\Big(\max_{y\in\R^n}|L(y,0,0)|+K_1+2K\Big)t,
\end{align*}
which implies
\begin{equation*}
|x-\gamma(0)|\le \Big(\max_{y\in\R^n}|L(y,0,0)|+K_1+2K\Big)t. 
\end{equation*}
Set $M_0:=\max_{y\in\R^n}|L(y,0,0)|+K_1+2K$. 
Combining Theorem \ref{thm:rep} with the above observation, 
we finish the proof.  
\end{proof}

\medskip
We give two important key ingredients, Lemmas  \ref{lem:sub}, \ref{lem:sup}, to prove Theorem \ref{thm1}. The key idea is that, due to the periodicity of $L$, (L2), we can obtain an estimate of the increment of solutions of \eqref{dotu}, see \eqref{w5} below for example.

\begin{lem}\label{lem:sub}
Let $M_0>0$. 
There exists $C>0$ which depends on $M_0$ such that 
\[
m((\sigma+l)t,0,(\sigma+l)y,(\sigma+l)c)
\le m(\sigma t,0,\sigma y,\sigma c)+m(lt,0,ly,lc)+C
\]
for any $\sigma,l>0$ with $\max\{\sigma t,lt\}\ge1$, $y\in B_{M_0t}(0)$. 
\end{lem}
\begin{proof}
Without any loss of generality, we assume $\sigma t\ge 1$. By Theorem \ref{thm:fundamental-sol}, 
we can take a minimizer $\gamma_\sigma\in\cC(0,\sigma y;\sigma t,0)$ of $m(\sigma t,0,\sigma y,\sigma c)$. 
Set 
\[
\xi^\sigma_1(s):=m(s,\gamma_\sigma(s),\sigma y,\sigma c)
\quad\text{for} \ s\in[0,\sigma t]. 
\]
Then, for $s, s'\in[0,\sigma t]$ with $s'>s$,
\[
\frac{\xi^\sigma_1(s')-\xi^\sigma_1(s)}{s'-s}=\frac{1}{s'-s}\int_s^{s'} L(\gamma_\sigma(s),\xi^\sigma_1(s),\dot{\gamma}_\sigma(s))\,ds. 
\]
Sending $s'\to s$ yields 
\begin{equation*}\label{ceq}
\begin{cases}
\dot{\xi}^\sigma_1(s)=L(\gamma_\sigma(s),\xi^\sigma_1(s),\dot{\gamma}_\sigma(s)) \quad \text{ for a.e. } s\in(0,\sigma t), \\ \xi^\sigma_1(0)=\sigma c. 
\end{cases}
\end{equation*}
Similarly, we can take a minimizer $\gamma_l \in\cC(0,ly;lt,0)$ of $m(lt,0,ly,lc)$, and set \[\xi^l_1(s):=m(s,\gamma_l(s),ly,lc).\]

Take $k\in\mathbb Z^n$ so that $k-ly\in Y:=[-\frac{1}{2},\frac{1}{2}]^n$. 
Set $\gamma_k(s):=\gamma_\sigma(s)+k$ for $s\in[0,\sigma t]$. 
Then, 
$\gamma_k(0)-(\sigma +l)y=\sigma y+k-(\sigma +l)y=k-ly\in Y$. 
We define $\eta:[0,(\sigma+l)t]\to\mathbb R^n$ by
\begin{equation*}
\eta(s):=
\begin{cases}
\alpha_1(s) \quad &\text{for} \ 0\le s\le \frac{1}{16},
\\ 
\gamma_k(s-\frac{1}{16}) \quad& \text{for} \ \frac{1}{16}\le s\le d+\frac{1}{16},
\\ 
\alpha_d(s-(d+\frac{1}{16})) \quad& \text{for} \ d+\frac{1}{16}\le s\le d+\frac{3}{16},
\\ 
\gamma_k(s+\frac{1}{16})\quad & \text{for} \ d+\frac{3}{16}\le s\le \sigma  t-\frac{1}{16},
\\ 
\alpha_2(s-(\sigma t-\frac{1}{16})) \quad&\text{for} \ \sigma t-\frac{1}{16}\le s\le \sigma t, 
\\ 
\gamma_l(s-\sigma t) \quad&\text{for} \ \sigma t \le s\le (\sigma +l)t, 
\end{cases}
\end{equation*}
where $d$ satisfying $0<d<\sigma t-\frac{1}{4}$ 
will be fixed later, $\alpha_d:[0,\frac{1}{8}]\to\R^n$ is a straight line satisfying $\alpha_d(0)=\gamma_k(d)$, $\alpha_d(\frac{1}{8})=\gamma_k(d+\frac{1}{4})$, and 
$\alpha_i:[0,\frac{1}{16}]\to\mathbb R^n$ for $i=1,2$ are straight lines satisfying 
$\alpha_1(0)=(\sigma +l)y$, 
$\alpha_1(\frac{1}{16})=\sigma y+k$, 
$\alpha_2(0)=k$, 
$\alpha_2(\frac{1}{16})=ly$, 
and 
$|\dot{\alpha}_i(s)|\le M_0$ for $s\in[0,\frac{1}{16}]$. 
Note that $\eta\in\cC(0,(\sigma +l)y;(\sigma +l)t,0)$ (see Figure \ref{Fig:1}).  
\begin{figure}[htb]
\centering
\includegraphics[width=8cm]{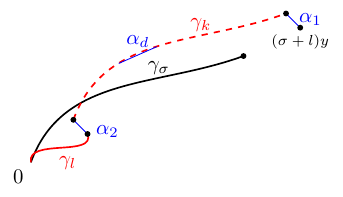}
\caption{Image of $\eta$.} \label{Fig:1}
\end{figure}

Let $\xi_2=\xi_{2,\eta}$ be the solution of 
\begin{equation*}
\begin{cases}
\dot{\xi_2}(s)=L(\eta(s),\xi_2(s),\dot{\eta}(s)) \quad \text{ for a.e. } s\in(0,(\sigma+l)t),
\\ 
\xi_2(0)=(\sigma +l)c. 
\end{cases}
\end{equation*}
We take $n_1\in\mathbb Z$ such that $lc+n_1-\xi_2(\sigma t)\in[0,1)$. Note that $\tilde \xi_2(s):=\xi_2(\sigma t+s)$ satisfies
\[
\dot {\tilde\xi}_2(s)=\dot{\xi}_2(\sigma t+s)=L(\eta(\sigma t+s),\xi_2(\sigma t+s),\dot\eta(\sigma t+s))=L(\gamma_l(s),\tilde \xi_2(s),\dot\gamma_l (s)) 
\]
for $s\in(0,lt)$ with the initial condition $\tilde{\xi}_2(0)=\xi_2(\sigma t)$. 
Let $\tilde{\xi}_1(s):=\xi^l_1(s)+n_1$. 
Noting that 
$\tilde\xi_1(0)=lc+n_1\ge \xi_2(\sigma t)=\tilde \xi_2(0)$, 
and the periodicity of $L$, (L2), and by the comparison principle for the following ODE
\[
\dot{X}(s)=F(X(s),s) \quad \text{ for a.e. } s\in(0,lt),
\]
where $F(X,s):=L(\gamma_l(s),X,\dot{\gamma}_l(s))$, we obtain 
\[
\tilde \xi_2(lt)=\xi_2((\sigma +l)t)\le \tilde \xi_1(lt)=\xi^l_1(lt)+n_1. 
\]
Also, by the choice of $n_1\in\N$, we have $lc+n_1< \xi_2(\sigma t)+1$, which implies
\begin{equation}\label{w5}
\xi_2((\sigma +l)t)-\xi_2(\sigma t)< \xi^l_1(lt)-lc+1.
\end{equation}
See Figure \ref{Fig:2} for the image. 
\begin{figure}[htb]
\centering
\includegraphics[width=15cm]{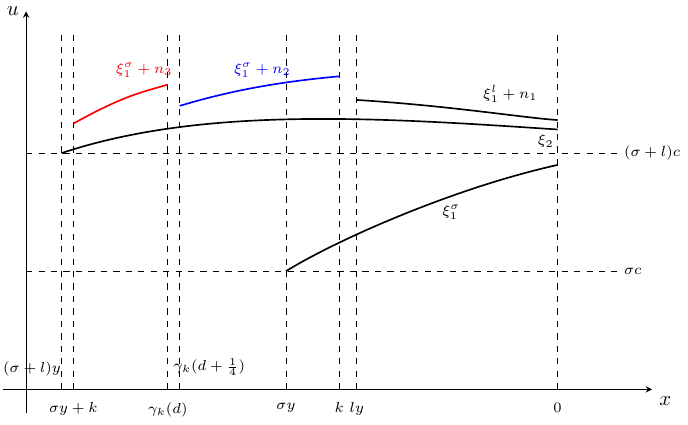}
\caption{Graphs of $\xi^\sigma_1$, $\xi^\sigma_1+n_2$, $\xi^\sigma_1+n_3$, $\xi^l_1+n_1$ and $\xi_2$.} \label{Fig:2}
\end{figure}

\smallskip
Now, we prove that there is $d\in\{0,\frac{1}{4},\frac{1}{2},\frac{3}{4},\dots,\lfloor \sigma t\rfloor-\frac{1}{4}\}$ such that
\begin{equation}\label{Cd}
\xi^\sigma_1\Big(d+\frac{1}{4}\Big)-\xi^\sigma_1(d)
=\int_d^{d+\frac{1}{4}}L(\gamma_\sigma(s),\xi^\sigma_1(s),\dot{\gamma}_\sigma(s))\, ds\le C_d,
\end{equation}
where
\[
C_d:=\max\bigg\{\max_{y\in\R^n,\,|v|\le M_0}|L(y,0,v)|+K,-\min\Big\{0,\min_{y\in\R^n,\,v\in\R^n}L(y,0,v)\Big\}+K\bigg\}>0.
\]
We take a straight line $\alpha:[0,\sigma t]\to\mathbb R^n$ connecting $\sigma y$ and $0$ with speed $|\dot{\alpha}|\le M_0$. 
Then, 
\begin{align*}
\xi^\sigma_1(\sigma t)-\sigma c&=m(\sigma t,0,\sigma y,\sigma c)-\sigma c\le \int_0^{\sigma t}L(\alpha(s),m(s,\alpha(s),y,c),\dot{\alpha}(s))\, ds
\\ &\le \Big(\max_{y\in\R^n,\,|v|\le M_0}L(y,0,v)+K\Big)\sigma t, 
\end{align*}
since $L(x,r,v)$ is periodic in $r$. If \eqref{Cd} does not hold, then
\[
\xi^\sigma_1(\lfloor \sigma t\rfloor)-\sigma c=\int_0^{\lfloor \sigma t\rfloor}L(\gamma_\sigma(s),\xi^\sigma_1(s),\dot{\gamma}_\sigma(s))\, ds\ge 4\lfloor \sigma t\rfloor   C_d,\]
which implies
\begin{align*}
C_d \sigma t\geq &\Big(\max_{y\in\R^n,\,|v|\le M_0}L(y,0,v)+K\Big)\sigma t
\\ &\ge \xi^\sigma_1(\sigma t)-\xi^\sigma_1(\lfloor \sigma t\rfloor)+\xi^\sigma_1(\lfloor \sigma t\rfloor)-\sigma c\ge \int_{\lfloor \sigma t\rfloor}^{\sigma t}L(\gamma_\sigma(s),\xi^\sigma_1(s),\dot{\gamma}_\sigma(s))\, ds+4\lfloor \sigma t\rfloor  C_d
\\ &\ge \min\Big\{0,\min_{y\in\R^n,\,v\in\R^n}L(y,0,v)\Big\}-K+4\lfloor \sigma t\rfloor  C_d\geq -C_d+4\lfloor \sigma t\rfloor  C_d.
\end{align*}
Thus, 
\[C_d(\lfloor \sigma t\rfloor+1)\geq C_d \sigma t\geq -C_d+4\lfloor \sigma t\rfloor  C_d,\]
that is, $2\geq 3\lfloor \sigma t\rfloor$, which leads to a contradiction since $\sigma t\ge 1$. By (L3), there is a constant $K_1>0$ such that 
\[
L(y,0,v)\ge |v|-K_1\quad
\text{for all} \ y\in\R^n,\ v\in\R^n.
\]
Therefore, we have
\begin{equation}\label{d1/4-d}
\begin{aligned}
C_d&\geq \int_d^{d+\frac{1}{4}}L(\gamma_\sigma(s),0,\dot{\gamma}_\sigma(s))\, ds-\frac{K}{4}
 \\ &\geq \int_d^{d+\frac{1}{4}}|\dot{\gamma}_\sigma(s)|\, ds-\frac{K_1+K}{4}\geq \bigg|\gamma_k\Big(d+\frac{1}{4}\Big)-\gamma_k(d)\bigg|-\frac{K_1+K}{4},
\end{aligned}
\end{equation}
which implies that $|\dot{\alpha}_d|\leq 8C_d+2(K_1+K)$. Hence, 
\begin{equation}\label{w4}
\begin{aligned}
\xi_2\Big(d+\frac{3}{16}\Big)-\xi_2\Big(d+\frac{1}{16}\Big)
&=\int_{d+\frac{1}{16}}^{d+\frac{3}{16}} L(\eta(s),\xi_2(s),\dot{\eta}(s))\, ds
\\ &\leq \int_{0}^{\frac{1}{8}} L(\alpha_d(s),0,\dot{\alpha}_d(s))\, ds+\frac{K}{8}\le M_1 
\end{aligned}
\end{equation}
for some $M_1>0$. 

\medskip
We choose $n_2\in\mathbb Z$ so that 
\[
\xi^\sigma_1\Big(d+\frac{1}{4}\Big)+n_2-\xi_2\Big(d+\frac{3}{16}\Big)\in[0,1).\]
Let $\xi_3(s)$ be the solution of 
\begin{equation*}
\begin{cases}
\dot{\xi}_3(s)
=
L(\eta(s),\xi_3(s),\dot{\eta}(s)) \quad\text{ for a.e. } s\in(d+\frac{3}{16},\sigma t), 
\\ 
\xi_3(d+\frac{3}{16})=\xi_1(d+\frac{1}{4})+n_2.
\end{cases}
\end{equation*}
Then, $\xi^\sigma_1(s+\frac{1}{16})+n_2=\xi_3(s)$ for $s\in[d+\frac{3}{16},\sigma t-\frac{1}{16}]$. By the comparison principle, we have
\[
\xi_2(\sigma t)\le \xi_3(\sigma t), 
\]
which implies
\[
\xi_2(\sigma t)-\xi_2\Big(d+\frac{3}{16}\Big)\le \xi_3(\sigma t)-\xi^\sigma_1\Big(d+\frac{1}{4}\Big)-n_2+1. 
\]
Since
\begin{align*}
&\xi_3(\sigma t)-\xi_3\Big(\sigma t-\frac{1}{16}\Big)\\
=&\, 
\int_{\sigma t-\frac{1}{16}}^{\sigma t} L(\eta(s),\xi_3(s),\dot{\eta}(s))\, ds
\le 
\int_0^{\frac{1}{16}} (L(\alpha_2(s),0,\dot{\alpha}_2(s))+K)\, ds
\le M_2
\end{align*}
for some $M_2>0$. We get
\[\xi_2(\sigma t)-\xi_2\Big(d+\frac{3}{16}\Big)
\le \xi_3\Big(\sigma t-\frac{1}{16}\Big)+M_2-\xi^\sigma_1\Big(d+\frac{1}{4}\Big)-n_2+1.
\]
Note that $\xi_3(\sigma t-\frac{1}{16})=\xi^\sigma_1(\sigma t)+n_2$, we get
\begin{equation}\label{w3}
\xi_2(\sigma t)-\xi_2\Big(d+\frac{3}{16}\Big)
\le 
\xi^\sigma_1(\sigma t)-\xi^\sigma_1\Big(d+\frac{1}{4}\Big)+M_2+1. 
\end{equation}

We choose $n_3\in\mathbb Z$ so that $\sigma c+n_3-\xi_2(\frac{1}{16})\in[0,1)$. 
By the comparison principle once again, we have
\[
\xi_2\Big(d+\frac{1}{16}\Big)\le \xi^\sigma_1(d)+n_3, 
\] 
which implies
\begin{equation}\label{w2}
\xi_2\Big(d+\frac{1}{16}\Big)-\xi_2\Big(\frac{1}{16}\Big)
\le 
\xi^\sigma_1(d)-\sigma c+1. 
\end{equation}
Also,
\begin{equation}\label{w1}
\xi_2\Big(\frac{1}{16}\Big)-(\sigma +l)c=\int_0^{\frac{1}{16}}L(\alpha_1(s),\xi_2(s),\dot{\alpha}_1(s))\, ds\le M_2.
\end{equation}
Combining \eqref{w5}, \eqref{w4}, \eqref{w3}, \eqref{w2}, \eqref{w1}, 
we get
\[
\xi_2((\sigma+l)t)-(\sigma +l)c
\le \xi^l_1(lt)+\xi^\sigma_1(\sigma t)+\xi^\sigma_1(d)-\xi^\sigma_1\Big(d+\frac{1}{4}\Big)-(\sigma +l)c+C.
\]
Note that
\[\xi^\sigma_1\Big(d+\frac{1}{4}\Big)-\xi^\sigma_1(d)
=
\int_d^{d+\frac{1}{4}}L(\gamma_\sigma(s),\xi^\sigma_1(s),\dot\gamma_\sigma(s))\, ds\ge \frac{1}{4}\Big(\min_{y\in\R^n,\,v\in\R^n}L(y,0,v)-K\Big). 
\]
By Lemma \ref{m<dotu}, we obtain 
\begin{align*}
m((\sigma +l)t,0,(\sigma +l)y,(\sigma +l)c)&\le \xi_2((\sigma +l)t)\le \xi^l_1(lt)+\xi^\sigma_1(\sigma t)+C
\\ &=m(lt,0,ly,lc)+m(\sigma t,0,\sigma y,\sigma c)+C
\end{align*}
for some $C>0$, which finishes the proof. 
\end{proof}

\begin{lem}\label{bdme}
Fix $M_0>0$. There is $C>0$ depending on $M_0$ and $H$ such that for all $\ep>0$, $t>0$, $x,y\in\R^n$ with $|x-y|\leq M_0t$ and $c\in\R$, we have
\[c-Ct\leq m^\ep(t,x,y,c)\leq c+Ct.\]
\end{lem}
\begin{proof}
Let $\alpha:[0,t/\ep]\to\R^n$ be a straight line with $\alpha(0)=y/\ep$ and $\alpha(t/\ep)=x/\ep$. Then $|\dot\alpha|\leq M_0$ and
\begin{align*}
m^\ep(t,x,y,c)&\leq \ep\Big(\frac{c}{\ep}+\int_0^{t/\ep}L\Big(\alpha(s),m\Big(s,\alpha(s),\frac{y}{\ep},\frac{c}{\ep}\Big),\dot\alpha(s)\Big)\, ds\Big)
\\ &\leq c+\Big(\max_{y\in\R^n,\,|v|\leq M_0}L(y,0,v)+K\Big)t.
\end{align*}
For all curves $\gamma\in \cC(x/\ep,y/\ep;t/\ep,0)$, we have
\[\ep\Big(\frac{c}{\ep}+\int_0^{t/\ep}L\Big(\gamma(s),m\Big(s,\gamma(s),\frac{y}{\ep},\frac{c}{\ep}\Big),\dot\gamma(s)\Big)\, ds\Big)
\geq c+\Big(\min_{y\in\R^n,\,v\in\R^n}L(y,0,v)-K\Big)t.
\]
By taking the infimum over all $\gamma\in \cC(x/\ep,y/\ep;t/\ep,0)$, we conclude.
\end{proof}

\begin{cor}\label{cor:m-bar}
For all $x,y\in\R^n$, $t\ge0$, $c\in\R$, 
the limit 
\[
\lim_{\ep\to0}m^\ep(t,x,y,c)
=\lim_{\ep\to 0}\ep m\Big(\frac{t}{\ep},\frac{x}{\ep},\frac{y}{\ep},\frac{c}{\ep}\Big)
\]
exists, which is denoted by $\ol{m}(t,x,y,c)$. 
\end{cor}
\begin{proof}
Without loss of generality, we may assume $x=0$.

We first consider the case where $t>0$. The proof is quite similar to that of the Fekete lemma. For each $y\in\R^n$, we can find $M_0>0$ such that $|y|\leq M_0t$. For $k>0$ and $c\in\R$, we set 
$\phi(k):=m(kt,0,ky,kc)+C$, 
where 
$C$ is a constant given by Lemma \ref{lem:sub}. 
Then, for $k,l>0$ with $\max\{kt,lt\}\geq 1$, by Lemma \ref{lem:sub}, 
\begin{equation}\label{phik}
\begin{aligned}
\phi(k+l)&=\, 
m\big((k+l)t,0,(k+l)y, (k+l)c\big)+C\\
&\le\, 
m(kt,0,ky,kc)+m(lt,0,ly,lc)+2C
=\phi(k)+\phi(l). 
\end{aligned}
\end{equation}

Fix $N>0$ large so that $Nt\ge1$, and we set 
\[\phi_{\inf}:=\inf_{k>N}\frac{\phi(k)}{k}.\]
By definition, for each $a>\phi_{\inf}$, there exists $l_0>N$ such that
\[\frac{\phi(l_0)}{l_0}\leq a.\]
For $k>0$, we take $n\in\mathbb N\cup\{0\}$ such that $nl_0<k\leq (n+1)l_0$. Similar to Lemma \ref{bdme}, there exists $M>0$ independent of $n$ such that $\phi(k-nl_0)\leq M$. Since $l_0t\geq 1$, 
by \eqref{phik}, we have 
\[\phi(k)\leq \phi(nl_0)+\phi(k-nl_0)\leq n\phi(l_0)+M,\]
which implies
\[\frac{\phi(k)}{k}< 
\frac{\phi(l_0)}{l_0}+\frac{M}{nl_0}\leq a+\frac{M}{nl_0}.\]
Now, sending $n\to\infty$ and $a\to \phi_{\inf}$, we conclude that the limit of $\frac{\phi(k)}{k}$ as $k\to\infty$ exists, and 
\[
\lim_{k\to\infty}\frac{\phi(k)}{k}=\phi_{\inf}=\inf_{k>N}\frac{\phi(k)}{k}.
\]

Next, we consider the case where $t=0$. By \eqref{mt0}, for all $\ep>0$, we have
\begin{equation*}
\ol{m}(0,x,y,c)=m^\ep(0,x,y,c)=
\begin{cases}
c\quad &\text{ if } x=y,\\  
\infty
\quad &\text{ if } x\neq y.
\end{cases}
\end{equation*}
This completes the proof.
\end{proof}

For $t\geq 0$ and $x\in\R^n$, we define
\begin{equation}\label{defu}
u(x,t):=\inf_{y\in\R^n,\, |x-y|\leq M_0t}\ol{m}(t,x,y,\varphi(y)),
\end{equation}
where $M_0>0$ is given by Proposition \ref{prop:uep=m}. 
It is clear that 
\[u(x,0)=\ol{m}(0,x,x,\varphi(x))=\lim_{\ep\to 0}m^\ep(0,x,x,\varphi(x))=\varphi(x).\]
\begin{lem}
There is $C>0$ depending on $M_0$ and $H$ such that for all $\ep>0$ and $t>0$, we have
\begin{equation}\label{u-uCt}
|u^\ep(x,t)-u(x,t)|\leq Ct.
\end{equation}
\end{lem}
\begin{proof}
For each $\delta>0$, we take $y_\delta\in\R^n$ with $|x-y_\delta|\leq M_0t$ such that $\ol{m}(t,x,y_\delta,\varphi(y_\delta))\leq u(x,t)+\delta$. Then, 
\[u^\ep(x,t)-u(x,t)\leq m^\ep(t,x,y_\delta,\varphi(y_\delta))-\ol{m}(x,y_\delta,\varphi(y_\delta))+\delta.\]
By Lemma \ref{bdme} and Corollary \ref{cor:m-bar}, that $\ol{m}$ is the limit of $m^\ep$, there is $C>0$ depending on $M_0$ and $H$ such that $\ol m(x,y_\delta,\varphi(y_\delta))\geq \varphi(y_\delta)-Ct$. We get
\[u^\ep(x,t)-u(x,t)\leq \varphi(y_\delta)+Ct-(\varphi(y_\delta)-Ct)+\delta=2Ct+\delta.\]
Sending $\delta\to 0$ yields $u^\ep(x,t)-u(x,t)\le 2Ct$. 

We take a point $y_\ep$ such that $u^\ep(x,t)=m^\ep(t,x,y_\ep,\varphi(y_\ep))$. By Proposition \ref{prop:uep=m}, $|y_\ep-x|\leq M_0t$. We have
\[u(x,t)-u^\ep(x,t)\leq \ol{m}(t,x,y_\ep,\varphi(y_\ep))-m^\ep(t,x,y_\ep,\varphi(y_\ep))\leq 2Ct.\]
This completes the proof.
\end{proof}

\begin{lem}\label{lem:sup}
Fix $M_0>0$. For $t\ge 1$ and $y\in \R^n$ with $|y|\le M_0t$, we have the following superadditivity property{\rm:} there is $C>0$ depending on $M_0$ such that
\[2m(t,0,y,c)\le m(2t,0,2y,2c)+C.\]
\end{lem}

\begin{proof}
We first take a minimizer $\gamma:[0,2t]\to\mathbb R^n$ of $m(2t,0,2y,2c)$. 
Set
\[w(s):=m(s,\gamma(s),2y,2c).\]
By using Burago's cutting lemma, 
we can take a collection of disjoint intervals $\{[a_i,b_i]\}_{1\le i\le k}\subset [0,2t]$ with $k\le\frac{n+3}{2}$ such that
\[
\sum_{i=1}^k \Big(\gamma(b_i)-\gamma(a_i),b_i-a_i,w(b_i)-w(a_i)\Big)=\Big(y,t,\frac{w(2t)-2c}{2}\Big). 
\]
Then, we consider a periodic shift $(\tilde\gamma,\tilde w)|_{[t_{i-1},t_i]}$ of $(\gamma,w)|_{[a_i,b_i]}$ satisfying 
\[\tilde\gamma(t^+_i)-\tilde\gamma(t^-_i)\in Y\quad \text{for all} \ i\in\{0,\dots,k\},\]
where
\[t_0:=0, \quad \tilde{\gamma}(t_0^-):=y,\quad \tilde\gamma(t_k^+):=0,\quad t_j:=\sum_{i=1}^j(b_i-a_i).\]
Then, we have $t_k=t$ and $|\tilde \gamma(t^-_k)|\le k\sqrt{n}$. 
Also, note that $\tilde\gamma(t^+_{i-1})-\gamma(a_i)\in\Z^n$, we set  
\begin{align*}\tilde{w}(s)&:=m(s-t_{i-1}+a_i,\tilde{\gamma}(s),2y+\tilde\gamma(t^+_{i-1})-\gamma(a_i),2c)+n_i
\\ &=m(s-t_{i-1}+a_i,\gamma(s-t_{i-1}+a_i),2y,2c)+n_i=w(s-t_{i-1}+a_i)+n_i,\quad s\in[t_{i-1},t_i],\end{align*} 
where $n_i\in\Z$ to be chosen later.

We first show the property of $\tilde\gamma$ similar to \eqref{Cd}. Since $|y|\le M_0t$, we take a straight line $\alpha_0:[0,2t]\to\R^n$ connecting $2y$ and $0$ to find
\begin{align*}
m(2t,0,2y,2c)-2c&\leq \int_0^{2t}L(\alpha_0(s),m(s,\alpha_0(s),2y,2c),\dot\alpha_0(s))\, ds
\\ &\leq \Big(\max_{y\in\R^n,\,|v|\le M_0}L(y,0,v)+K\Big)2t,
\end{align*}
which implies
\begin{align*}
&\Big(\max_{y\in\R^n,\,|v|\le M_0}L(y,0,v)+K\Big)2t
\\ &\geq \int_{\cup_{i=1}^k[a_i,b_i]}L(\gamma(s),w(s),\dot\gamma(s))\, ds+\int_{[0,2t]\backslash \cup_{i=1}^k[a_i,b_i]}L(\gamma(s),w(s),\dot\gamma(s))\, ds
\\ &\geq \sum_{i=1}^k\int_{a_i}^{b_i}L(\gamma(s),w(s),\dot\gamma(s))\, ds+\Big(\min_{y\in\R^n,\,|v|\in\R^n}L(y,0,v)-K\Big)t
\\ &=\int_0^t L(\tilde \gamma(s),\tilde w(s),\dot{\tilde\gamma}(s))\, ds+\Big(\min_{y\in\R^n,\,|v|\in\R^n}L(y,0,v)-K\Big)t,
\end{align*}
where for the second inequality, we used $\sum_{i=1}^k(b_i-a_i)=t$, for the last equality, we recall that $(\tilde\gamma,\tilde w)|_{[t_{i-1},t_i]}$ is a $\Z^{n+1}$-periodic shift of $(\gamma,w)|_{[a_i,b_i]}$. Similar to \eqref{Cd}, we can find $d\in\{0,\frac{1}{4},\frac{1}{2},\frac{3}{4},\dots,\lfloor t\rfloor-\frac{1}{4}\}$ such that
\begin{equation*}
\int_d^{d+\frac{1}{4}}L(\tilde \gamma(s),\tilde w(s),\dot{\tilde \gamma}(s))\, ds\le C_d,
\end{equation*}
for some $C_d>0$. 
We take $i_d,\bar i_d\in\{0,1,\dots,k\}$ with $i_d\leq \bar i_d-1$ satisfying
\[
t_{i_d}\leq d< d+\frac{1}{4}\leq t_{\bar i_d}. 
\] 
Here, if $i_d=\bar i_d-1$, then $\tilde\gamma(d)$ and $\tilde\gamma(d+\frac{1}{4})$ belong to the same continuous piece of $\tilde\gamma$. Similar to \eqref{d1/4-d}, we can show that
\[
\int_d^{d+\frac{1}{4}}|\dot{\tilde\gamma}(s)|\, ds
\]
is bounded by a constant $M_d$. Note that $\tilde\gamma(t^+_i)-\tilde\gamma(t^-_i)\in Y$ for all $i\in\{0,\dots,k\}$, we have \[\left|\tilde \gamma\left(d+\frac{1}{4}\right)-\tilde\gamma(d)\right|\leq M_d+k\sqrt{n}.\] Define
\begin{equation*}
\eta_0(s):=
\begin{cases}
\tilde\gamma(s) \quad&\text{for} \ 0\le s\le d,
\\ \alpha_d(s-d) \quad& \text{for} \ d\le s\le d+\frac{1}{8},
\\ \tilde \gamma(s+\frac{1}{8}) \quad& \text{for} \ d+\frac{1}{8}\le s\le t-\frac{1}{8},
\end{cases}
\end{equation*}
where $\alpha_d:[0,\frac{1}{8}]\to\R^n$ is a straight line satisfying $\alpha_d(0)=\tilde \gamma(d)$ and $\alpha_d(\frac{1}{8})=\tilde \gamma(d+\frac{1}{4})$. It is clear that $|\dot{\alpha}_d|$ is bounded. Here, if $d=t_{i_d}$ (resp. $d+\frac{1}{4}=t_{\bar i_d}$), we take $\alpha_d$ satisfying $\alpha_d(0)=\tilde\gamma(t_{i_d}^+)$ (resp. $\alpha_d(\frac{1}{8})=\tilde\gamma(t_{\bar i_d}^-)$). The discontinuous points of $\eta_0$ are denoted by $\{s_0,\dots,s_{i_d},s_{\bar i_d},\dots,s_k\}$ in this order, where
\begin{equation*}
s_i=
\begin{cases}
t_i \quad&\text{for} \ 0\le t_i\leq d,
\\ t_i-\frac{1}{8}\quad& \text{for} \ d+\frac{1}{4}\leq t_i\le t,
\end{cases}
\end{equation*}
and we recall that $t_{i_d}\leq d$, $t_{\bar i_d}\geq d+\frac{1}{4}$. 
We then connect $\eta_0(s^\pm_i)$ by straight lines $\alpha_i$ to construct a continuous curve $\eta:[0,t]\to\mathbb R^n$ with $\eta(0)=y$ and $\eta(t)=0$. More precisely, we denote by $\{\bar s_i\}_{0\leq i\leq k'}$ with $k'=k-\bar i_d+i_d+1$ the set of discontinuous points of $\eta_0$ in order, and for $i\in\{0,\dots,k'\}$, $\alpha_i:[0,\frac{1}{8(k'+1)}]\to\R^n$ are straight lines, and we set
\begin{equation*}
\begin{cases}
\eta(s+\bar s_i+\frac{i}{8(k'+1)})=\alpha_i(s)\quad&\text{for}\ s\in[0,\frac{1}{8(k'+1)}],
\\ \eta(s+\frac{i+1}{8(k'+1)})=\eta_0(s)\quad&\text{for} \ \bar s_i\le s\le \bar s_{i+1}.
\end{cases}
\end{equation*}
See Figure \ref{Fig:3}. 
By the construction of $\tilde\gamma$, we know that $\{\dot\alpha_i\}_{0\leq i\leq k'}$ are bounded.
\begin{figure}[htb]
\centering
\includegraphics[width=11cm]{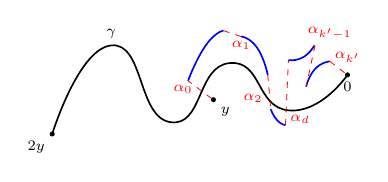}
\caption{Image of $\eta$.} \label{Fig:3}
\end{figure}

\smallskip
Now, we consider the solution $v(s)$ of
\begin{equation*}
\begin{cases}
\dot v(s)=L(\eta(s),v(s),\dot\eta(s))\quad \text{for} \ s\in[0,t],
\\ v(0)=c.
\end{cases}
\end{equation*}
Without loss of generality, we assume $i_d>1$. We choose suitable $n_1\in\mathbb Z$ such that
\[w(a_1)+n_1-v\Big(\frac{1}{8(k'+1)}\Big)\in[0,1).\]
Define $\tilde v(s)=v(s+\frac{1}{8(k'+1)})$, we have
\begin{equation*}
\begin{cases}
\dot{\tilde v}(s)=L(\tilde \gamma(s),\tilde v(s),\dot{\tilde \gamma}(s))\quad \text{for} \ s\in[0,t_1],
\\ \tilde v(0)=v\left(\frac{1}{8(k'+1)}\right).
\end{cases}
\end{equation*}
Note that we also have
\begin{equation*}
\begin{cases}
\dot{\tilde w}(s)=L(\tilde \gamma(s),\tilde w(s),\dot{\tilde \gamma}(s))\quad \text{for} \ s\in[0,t_1],
\\ \tilde w(0^+)=w(a_1)+n_1.
\end{cases}
\end{equation*}
By the comparison principle of the following ODE
\[\dot X=F(X(s),s),\quad \text{for\ a.e.\ }s\in [0,t_1],\]
where $F(X,s):=L(\tilde\gamma(s),X,\dot{\tilde\gamma}(s))$, we have $\tilde v(t_1)\leq \tilde w(t_1^-)$. Recall that $t_1=s_1$ since we assume $i_d>1$, we get
\[v\Big(s_1+\frac{1}{8(k'+1)}\Big)\leq w(b_1)+n_1,
\]
which implies
\begin{equation}\label{v1}
v\Big(s_1+\frac{1}{8(k'+1)}\Big)-v\Big(\frac{1}{8(k'+1)}\Big)\le w(b_1)-w(a_1)+1.
\end{equation}
Similarly, we choose $n_i\in\mathbb Z$ such that $w(a_i)+n_i-v(s_{i-1}+\frac{i}{8(k'+1)})\in[0,1)$ for $2\le i\le i_d$. Then we get
\begin{equation}\label{v2}
v\Big(s_{i}+\frac{i}{8(k'+1)}\Big)-v\Big(s_{i-1}+\frac{i}{8(k'+1)}\Big)\le w(b_{i})-w(a_{i})+1.
\end{equation}
We also have
\begin{equation}\label{v3}
v\Big(d+\frac{i_d+1}{8(k'+1)}\Big)-v\Big(s_{i_d}+\frac{i_d+1}{8(k'+1)}\Big)\le w(d+a_{i_d+1}-t_{i_d})-w(a_{i_d+1})+1, 
\end{equation}
and
\begin{equation}\label{v4}
v\Big(s_{\bar i_d}+\frac{i_d+1}{8(k'+1)}\Big)-v\Big(d+\frac{1}{8}+\frac{i_d+1}{8(k'+1)}\Big)\le w(b_{\bar i_d})-w\Big(d+\frac{1}{4}+b_{\bar i_d}-t_{\bar i_d}\Big)+1.
\end{equation}
Here we note that, since we connect $\tilde\gamma(d)$ and $\tilde\gamma(d+\frac{1}{4})$ directly by $\alpha_d$, there are $\bar i_d-i_d-1$ discontinuous points we do not need to connect. For $\bar i_d\le i\le k-1$, we have
\begin{equation}\label{v5}
v\Big(s_{i+1}+\frac{i-\bar i_d+i_d+2}{8(k'+1)}\Big)-v\Big(s_{i}+\frac{i-\bar i_d+i_d+2}{8(k'+1)}\Big)\le w(b_{i+1})-w(a_{i+1})+1.
\end{equation}
The proof of \eqref{v2}-\eqref{v5} is similar to that of \eqref{v1}, we omit it. We have
\begin{align*}
\sum_{i=i_d+1}^{\bar i_d}(w(b_i)-w(a_i))&=\sum_{i=i_d}^{\bar i_d-1}(\tilde w(t_{i+1}^-)-\tilde w(t_i^+))
\\ &=\tilde w(t_{\bar i_d}^-)-\tilde w\Big(d+\frac{1}{4}\Big)+\tilde w(d)-\tilde w(t_{i_d}^+)+\int_{d}^{d+\frac{1}{4}}L(\tilde\gamma(s),\tilde w(s),\dot{\tilde\gamma}(s))\, ds
\\ & \geq w(b_{\bar i_d})-w\Big(d+\frac{1}{4}+b_{\bar i_d}-t_{\bar i_d}\Big)+w(d+a_{i_d+1}-t_{i_d})-w(a_{i_d+1})
\\ &\hspace{6.8cm} +\frac{1}{4}\Big(\min_{y\in\R^n,\,v\in\R^n}L(y,0,v)-K\Big).
\end{align*}
Similar to \eqref{w4}, we get
\[
v\Big(d+\frac{1}{8}+\frac{i_d+1}{8(k'+1)}\Big)-v\Big(d+\frac{i_d+1}{8(k'+1)}\Big)\leq \int_0^{\frac{1}{8}}L(\alpha_d,0,\dot\alpha_d(s))\, ds+\frac{K}{8}\leq M_1
\]
for some $M_1>0$, since $\dot\alpha_d$ is bounded. Note that for $i\in\{0,\dots,k'\}$,
\[v\Big(\bar s_i+\frac{i+1}{8(k'+1)}\Big)-v\Big(\bar s_i+\frac{i}{8(k'+1)}\Big)=\int_0^{\frac{1}{8(k'+1)}}L\bigg(\alpha_i(s),v\Big(s+\bar s_i+\frac{i}{8(k'+1)}\Big),\dot\alpha_i(s)\bigg)\, ds\]
are bounded since $\{\dot\alpha_i\}_{0\le i\le k'}$ are bounded. Here we recall that $\bar s_{k'}=s_k=t-\frac{1}{8}$ and $k'=k-\bar i_d+i_d+1$.

\begin{figure}[H]
\centering
\includegraphics[width=16.5cm]{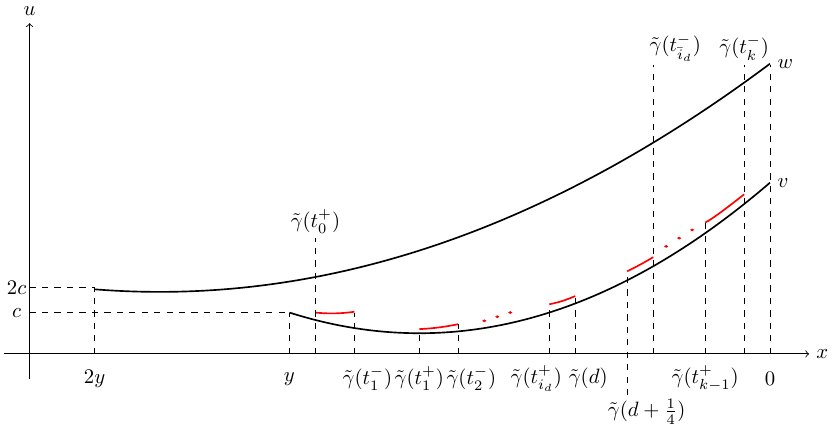}
\caption{Graphs of $v$, $w$, and periodic shifts of parts of the graph of $w$.} \label{Fig:4}
\end{figure}

Summing up \eqref{v1}--\eqref{v5} and using Lemma \ref{m<dotu}, we get
\begin{align*}
m(t,0,y,c)-c&\le v(t)-c\le \sum_{i=1}^k (w(b_i)-w(a_i))+C
\\ &=\frac{w(2t)-2c}{2}+C=\frac{1}{2}m(2t,0,2y,2c)-c+C
\end{align*}
for some $C>0$.
\end{proof}

\begin{cor}\label{cor:diff-m}
Let $m^\ep$ and $\overline{m}$ be the functions defined in Corollary {\rm\ref{cor:m-bar}}. 
Let $M_0>0$. There is $C>0$ which depends on $M_0$ such that 
\[
\bigg|m^\ep(t,x,y,c)-\ol{m}(t,x,y,c)\bigg|\le C\ep
\]
for all $\ep>0$, $t\ge\ep$, and $x,y\in\R^n$ with 
$|x-y|\le M_0t$. 
\end{cor}
This is a straightforward result of Lemmas \ref{lem:sub}, \ref{lem:sup}. For a similar proof, one can refer to  \cite[Theorem 4.2]{HJMT}.

\bigskip
We finally give a proof of Theorem \ref{thm1}. 
\begin{proof}[Proof of Theorem {\rm\ref{thm1}}] 
Here we recall the definition of $u(x,t)$, which is given by \eqref{defu}. Note that $u^\ep(x,0)=u(x,0)=\varphi(x)$. If $0<t\le \ep$, by \eqref{u-uCt} we have \[|u^\ep(x,t)-u(x,t)|\le Ct\le C\ep.\] We only need to consider the case where $t>\ep$.

On one hand, let $y_\ep$ be a point such that $u^\ep(x,t)=m^\ep(t,x,y_\ep,\varphi(y_\ep))$. By Proposition \ref{prop:uep=m}, we have $|y_\ep-x|\le M_0t$. 
By Corollary \ref{cor:diff-m}, we obtain 
\[u(x,t)\le \ol{m}(t,x,y_\ep,\varphi(y_\ep))\le \varepsilon m\Big(\frac{t}{\varepsilon},\frac{x}{\varepsilon},\frac{y_\ep}{\varepsilon},\frac{\varphi(y_\ep)}{\varepsilon}\Big)+C\ep=u^\ep(x,t)+C\ep.\]

On the other hand, for all $\delta>0$, let $y_\delta$ with $|y_\delta-x|\le M_0t$ be a point such that $\overline{m}(t,x,y_\delta,\varphi(y_\delta))\leq u(x,t)+\delta$, we get
\[
u^\ep(x,t)\le \varepsilon m\Big(\frac{t}{\varepsilon},\frac{x}{\varepsilon},\frac{y_\delta}{\varepsilon},\frac{\varphi(y_\delta)}{\varepsilon}\Big)
\le \ol{m}(t,x,y_\delta,\varphi(y_\delta))+C\ep
\le u(x,t)+\delta+C\ep.
\]
Let $\delta\to 0$ we get $u^\ep(x,t)\le u(x,t)+C\ep$.

The rate $O(\ep)$ is optimal, since (H1) includes the case where $H(y,r,p)$ is independent of $r$.
See \cite[Proposition 4.3]{MTY} for an example.
\end{proof}

\section{Effective Lagrangian, Hamiltonian, and correctors}\label{sec:effective}
Our main goal in this section is to give the proof of Theorem \ref{thm:qualitative}.

\subsection{Effective Lagrangian and Hamiltonian}
\begin{lem}\label{lem:form}
Let $\ol m$ be defined as in Corollary {\rm\ref{cor:m-bar}}.
Then, for $x,y\in \R^n$, $t>0$, and $c\in \R$, we can rewrite as 
\[
\overline{m}(t,x,y,c)=c+t\overline{L}\left(\frac{x-y}{t}\right),
\]
for a convex function $\overline{L}:\R^n \to \R$ which is superlinear.
\end{lem}
\begin{proof}
By periodicity we have, for $k\in\mathbb Z^n$, $x,y\in \R^n$, $t>0$, and $c\in \R$,
\[
\begin{cases}
m(t,x,y,c+1)=m(t,x,y,c)+1,\\
m(t,x+k,y+k,c)=m(t,x,y,c).    
\end{cases}
\]
For each $c\in\mathbb R\backslash\{0\}$, we take a sequence $\ep_n\to 0$ such that $c/\ep_n\in\mathbb Z$, i.e., $\ep_n=|c|/N_n$ for $N_n\in\mathbb N$ and $N_n\to \infty$. Then,
\[\ep_n m\Big(\frac{t}{\ep_n},\frac{x}{\ep_n},\frac{y}{\ep_n},\frac{c}{\ep_n}\Big)=\ep_nm\Big(\frac{t}{\ep_n},\frac{x}{\ep_n},\frac{y}{\ep_n},0\Big)+c\to \overline{m}(t,x,y,0)+c.\]
Hence, we only need to consider $\overline{m}(t,x,y,0)$. 
Noting further that
\[\lim_{\ep\to 0}\ep m\Big(\frac{t}{\ep},\frac{x}{\ep},\frac{y}{\ep},0\Big)=t\lim_{\ep\to 0}\frac{\ep}{t} m\Big(\frac{t}{\ep},\frac{x}{\ep},\frac{y}{\ep},0\Big)=t\lim_{\tau\to 0}\tau m\Big(\frac{1}{\tau},\frac{xt^{-1}}{\tau},\frac{yt^{-1}}{\tau},0\Big),\]
we only need to consider the limit of
\[
\ep m\Big(\frac{1}{\ep},\frac{\tilde x}{\ep},\frac{\tilde y}{\ep},0\Big)
\quad \text{for each} \ 
\tilde x,\tilde y\in\mathbb R^n. 
\]
If $\tilde y$ is rational, we can take a sequence $\ep_n\to 0$ such that $\tilde y/\ep_n\in\mathbb Z^n$, i.e., we take $\ep_n=|y_i|/N^n_i$ for $N^n_i\in\mathbb N$ and $N^n_i\to \infty$. We have
\[\ep_n m\Big(\frac{1}{\ep_n},\frac{\tilde x}{\ep_n},\frac{\tilde y}{\ep_n},0\Big)=\ep_n m\Big(\frac{1}{\ep_n},\frac{\tilde x-\tilde y}{\ep_n},0,0\Big),\]
which implies
\[\overline{m}(1,\tilde x,\tilde y,0)=\overline{m}(1,\tilde x-\tilde y,0,0),\]
where $\tilde y$ is rational. 
Using the continuity of $m^\ep$ and Corollary \ref{cor:diff-m}, i.e.,  $|m^\ep(t,x,y,c)-\overline{m}(t,x,y,c)|\leq C\ep$, we get the continuity of $\overline{m}$. Since rational vectors are dense in $\R^n$, we can extend the above equality to the whole space, which implies
\[
\overline{m}(1,\tilde x,\tilde y,0)=\overline{m}(1,\tilde x-\tilde y,0,0)=:\overline{L}(\tilde x-\tilde y)
\quad\text{for all} \ \tilde{x}, \tilde{y}\in\R^n. 
\]
Letting $\tilde x=x/t$ and $\tilde y=y/t$, we obtain
\[
\overline{m}(t,x,y,0)=
t\overline{m}(1,\tilde{x},\tilde{y},0)=
t\overline{L}\Big(\frac{x-y}{t}\Big).
\]

Now, we show the convexity of $\overline{L}$. 
Fix $x_1, x_2\in\R^n$, and take $M_0>0$ so that $|x_1|+|x_2|\le M_0$. 
By repeating the argument in the proof of Lemma \ref{lem:sub}, we can prove that there exists $C>0$ such that for all $\ep\in(0,1)$, $\lambda\in(0,1)$ and $x_1,x_2\in \mathbb R^n$, we have
\begin{align*}
&\ep m\Big(\frac{1}{\ep},\frac{\lambda x_1+(1-\lambda)x_2}{\ep},0,0\Big)
\\ &\leq \ep m\Big(\frac{\lambda}{\ep},\frac{\lambda x_1}{\ep},0,0\Big)+\ep m\Big(\frac{(1-\lambda)}{\ep},\frac{(1-\lambda)x_2}{\ep},0,0\Big)+C\ep.
\end{align*}
Sending $\ep\to 0$, we get
\begin{align*}
\overline{L}(\lambda x_1+(1-\lambda)x_2)&\leq\,
\overline{m}(\lambda,\lambda x_1,0,0)+\overline{m}(1-\lambda,(1-\lambda) x_2,0,0)\\
&=\, 
\lambda\overline{m}(1,x_1,0,0)+(1-\lambda)\overline{m}(1,x_2,0,0)
=\lambda \overline{L}(x_1)+(1-\lambda)\overline{L}(x_2).    
\end{align*}
Hence, $\overline L$ is convex.

We finally prove the superlinearity of $\overline{L}$. 
Because of the superlinearity of $L$, for all $A>0$, there is $K(A)\in\mathbb R$ such that $L(y,0,v)\ge A|v|-K(A)$ for all $y,v\in\R^n$.
Thus, for a minimizer $\eta$ of $m^\ep(1,x,y,0)$ for fixed $x,y\in\R^n$,
\begin{align*}
\ep m\Big(\frac{1}{\ep},\frac{x}{\ep},\frac{y}{\ep},0\Big)&= \ep \int_0^{1/\ep}L\Big(\eta(s),m\Big(s,\eta(s),\frac{y}{\ep},0\Big),\dot\eta(s)\Big)\, ds
\\ &\geq \ep \int_0^{1/\ep}A|\dot\eta(s)|\, ds-(K+K(A))\geq A|x-y|-(K+K(A)).
\end{align*}
Sending $\ep\to 0$, we conclude
\[
\overline{L}(x-y)\geq A|x-y|-(K+K(A)) 
\quad\text{for all} \ x,y\in\R^n.
\]
\end{proof}

Thanks to Lemma \ref{lem:form}, the limit defined in \eqref{defu} is of the form
\[u(x,t)=\min_{y\in\R^n,\,|x-y|\leq M_0 t}\bigg\{\varphi(y)+t\overline{L}\Big(\frac{x-y}{t}\Big)\bigg\}.\]
Moreover, we can define the effective Hamiltonian $\overline{H}:\R^n\to\R$ by the Legenedre transform of $\overline{L}$, 
i.e., $\overline{H}(p):=\sup_{v\in\R^n}\{p\cdot v-\overline{L}(v)\}$. 
It is now clear to see that $u$ is the unique solution to \eqref{e0}. 

\subsection{Existence of bounded continuous correctors}
Once again, we would emphasize that our argument to prove homogenization both quantitatively and qualitatively does not rely on the existence of any correctors or sub/super correctors. 
Here, interestingly, we establish the existence of bounded continuous correctors by using our main theorem, Theorem \ref{thm1}. 
It is our hope that correctors will find useful applications in the future.

\begin{lem}\label{lem:corr-w}
For any $p\in\R^n$ and $w_0\in\BUC(\R^n)\cap\Lip(\R^n)$, let $w$ be the unique solution to \eqref{Hp} in $\BUC(\R\times[0,\infty))$.
Then, there is a constant $C$ depending only on $H$, $p$ and $\|w_0\|_{L^\infty(\R^n)}$ such that 
\[
|w(y,\tau)+\overline{H}(p)\tau|\leq C
\quad \text{for\ all} \ y\in\R^n, \ \tau>0.  
\]
\end{lem}
\begin{proof}
We first assume $w_0(y)=0$. 
For $(x,t)\in \R^n\times [0,\infty)$, let
\[u^\ep(x,t):=p\cdot x+\ep w\Big(\frac{x}{\ep},\frac{t}{\ep}\Big).\]
Then, we have
\begin{equation*}
\begin{cases}
u^\varepsilon_t+H(\frac{x}{\varepsilon},\frac{u^\varepsilon}{\varepsilon},Du^\varepsilon)=0
\quad &\text{for} \ x\in\R^n,\ t>0,
\\ 
u^\varepsilon(x,0)=p\cdot x
\quad &\text{for} \ x\in\R^n.  
\end{cases}
\end{equation*}
By Theorem \ref{thm1}, we have $\|u^\ep-u\|_{L^\infty(\R^n\times [0,\infty))}\leq C\ep$, where $u$ solves \eqref{e0} with $u(x,0)=p\cdot x$. It is clear to see that $u=p\cdot x-\overline{H}(p)t$. Let $y=x/\ep$ and $\tau=t/\ep$. We have
\[|\ep w(y,\tau)+\overline{H}(p)t|=|p\cdot x+\ep w(y,\tau)-(p\cdot x-\overline{H}(p)t)|=|(u^\ep-u)(x,t)|\leq C\ep,\]
which implies that
\[|w(y,\tau)+\overline{H}(p)\tau|\leq C.\]

Next, we consider general $w_0\in\BUC(\R^n)\cap\Lip(\R^n)$. Take $N\in\N$ such that 
\[\|w_0\|_{L^\infty(\R^n)}\leq N.\] 
Let  $w^0$ be the solution of \eqref{Hp} with $w(y,0)=0$. 
Similar to \eqref{eks}, we have 
\[
|w^0_\tau(y,\tau)|\leq \sup_{y\in\R^n,\, r\in\R}|H(y,r,0)|e^{K\tau} \quad \text{ for a.e. } (y,\tau)\in \R^n\times [0,\infty),
\]
which implies that $w^0$ is Lipschitz continuous on $\R^n\times [0,T]$ with a Lipschitz constant depending on $T$ for each $T>0$.
Then $w^0+N$ (resp., $w^0-N$) solves \eqref{Hp} with the initial data $N$ (resp., $-N$). Since $w^0$ is locally Lipschitz continuous, by the comparison principle, we have $w^0-N\leq w\leq w^0+N$, which implies
\[|w(y,\tau)+\overline{H}(p)\tau|\leq |w(y,\tau)-w^0(y,\tau)|+|w^0(y,\tau)+\overline{H}(p)\tau|\leq C+N.\]
This completes the proof.
\end{proof}

\begin{proof}[Proof of Theorem {\rm\ref{thm:qualitative}}]
    The first part follows directly from Lemma \ref{lem:form}.
    The second part follows from Lemma \ref{lem:corr-w}. 
    We also note that $v(y,\tau)=w(y,\tau)+\overline{H}(p)\tau$ solves \eqref{cell}.
    Further, $v$ is bounded and continuous on $\R^n\times [0,\infty)$.
\end{proof}

\subsection{Representation formulas of the effective Hamiltonian}\label{ssec:rep-formula}
We establish some representation formulas of the effective Hamiltonian $\overline{H}(p)$ defined by \eqref{def:eff-H} in this subsection. 

\medskip

 It is natural to ask whether the classical inf-sup formula for the effective Hamiltonian holds. 
 We set $p=0$ in \eqref{Hp}. 
 By \cite[Theorem 1.4]{NWY}, there are two finite constant $c_1,c_2$ with $c_1\le c_2$ depending on $H$ such that
 \[
 H(y,v,D_yv)=c\quad \text{ for } y \in\R^n
 \]
has $\mathbb Z^n$-periodic solutions if and only if $c\in[c_1,c_2]$. As pointed out in \cite{CLQY,NWY}, for 
\[
H(y,r,p)=\frac{p^2}{2}-\cos r \quad \text{ for } (y,r,p)\in \R^n\times \R \times \R^n,
\] 
we have $[c_1,c_2]=[-1,1]$, which is not a single point. 
Now, if $0\in[c_1,c_2]$, we know that $\overline{H}(0)=0$, and the correctors corresponding to \eqref{cell} are exactly solutions of 
\[
H(y,v,D_y v)=0,
\]
which is time-independent. 
It is clear that
\[\inf_{\substack{v\in C^\infty(\R^n\times (0,\infty))\\ \|v\|_{C^1}<\infty}}\sup_{(y,\tau)\in\R^n\times(0,\infty)}(v_\tau+H(y,v,D_yv))\leq \inf_{v\in C^\infty(\T^n)}\sup_{y\in\T^n}H(y,v,D_yv).\]
Since $[c_1,c_2]$ may not reduce to a point, the following situation may happen
\[c_1=\inf_{v\in C^\infty(\T^n)}\sup_{y\in\T^n}H(y,v,D_yv)<\overline{H}(0)=0,\]
where the first equality was proved in \cite{WY2}. 
Thus, any such inf-sup formula has to be properly adjusted correspondingly to our setting.

\begin{lem}\label{lem:rep-H-bar-1}
    Assume {\rm (H1)--(H5)}.
For $p\in \R^n$, the effective Hamiltonian $\overline{H}(p)$ has a representation formula of the form: 
    \begin{align*}
        \ol H(p) = \inf\bigg\{-\limsup_{\tau \to \infty} &\sup_{y\in \R^n}\frac{\phi(y,\tau)}{\tau} \mid  
        \\ & \phi \in \AL(\R^n\times [0,\infty)) \text{ is a subsolution to } \eqref{Hp}\ \text{with}\ \phi(y,0)<\infty \bigg\}.
    \end{align*}
\end{lem}

\begin{proof}
    Let $\phi \in \AL(\R^n\times [0,\infty))$ be a subsolution to  \eqref{Hp} with $\phi(y,0)<\infty$.  We take $N\in\N$ large such that $\phi(y,0)-N\leq 0$ for all $y\in \R^n$. 
    Let $w^0$ be the solution to \eqref{Hp} with $w(y,0)=0$, which is locally Lipschitz continuous. Then, by the comparison principle,
    \[
    \phi(y,\tau)-N \leq w^0(y,\tau) \quad \text{ for } (y,\tau)\in \R^n\times [0,\infty).
    \]
    Hence,
    \[
    -\limsup_{\tau \to \infty} \sup_{y\in \R^n}\frac{\phi(y,\tau)}{\tau}=-\limsup_{\tau \to \infty} \sup_{y\in \R^n}\frac{\phi(y,\tau)-N}{\tau} \geq -\limsup_{\tau \to \infty} \sup_{y\in \R^n}\frac{w^0(y,\tau)}{\tau}= \ol H(p).
    \]
    On the other hand, by picking $\phi=w^0$, we get the equality.    
\end{proof}

We next give a different type of representation formula for $\ol H(p)$. 
For $\phi\in \BUC(\R^n\times [0,\infty))$, we define
\[
S_\phi = \left\{ c\in \R \mid \phi_\tau + H(y,p\cdot y+\phi-c\tau, D\phi) \leq c \text{ in } \R^n\times (0,\infty)\right\},
\]
where the above inequality holds in the viscosity sense. It is clear that $S_\phi \neq \emptyset$ if $\phi \in C^1(\R^n\times [0,\infty))$ with $\|\phi\|_{C^1(\R^n\times [0,\infty))} <+\infty$.
If $S_\phi=\emptyset$, we set $\inf S_\phi=+\infty$. 

\begin{prop}\label{prop:inf-sup}
     Assume {\rm (H1)--(H5)}. 
    Fix $p\in \R^n$.
    Then, we have the following inf-sup representation formula
    \[
    \ol H(p)= \inf_{\phi\in \BUC(\R^n\times [0,\infty))}  \inf S_\phi.
    \]
\end{prop}

\begin{proof}
    Fix $\phi\in \BUC(\R^n\times [0,\infty))$. If $S_\phi\neq \emptyset$,
    we pick $c\in S_\phi$ and $m\in \N$ such that
    \[
    \phi(y,0) \leq m \quad \text{ for } y\in \R^n.
    \]
    Letting $\psi(y,\tau)=\phi(y,\tau)-c\tau -m$ for $(y,\tau)\in \R^n\times [0,\infty)$, 
    we see that $\psi$ is a subsolution to \eqref{Hp} with $\psi(y,0)\leq 0$. By the comparison principle, we obtain 
    \[
    \psi(y,\tau)=\phi(y,\tau)-c\tau -m \leq w^0(y,\tau) \quad \text{ for } (y,\tau)\in \R^n\times [0,\infty),
    \]
    where $w^0$ is the solution of \eqref{Hp} with $w(y,0)=0$. Then, by Lemma \ref{lem:corr-w}, 
    \[
    (\ol H(p)-c)\tau \leq (w^0(y,\tau)+\ol H(p)\tau) + m -\phi(y,\tau) \leq C,
    \]
    for some $C>0$ indepedent of $\tau$. This implies $\ol H(p) \leq c$. For $S_\phi=\emptyset$, it is clear that $\ol{H}(p)\leq \inf S_\phi=+\infty$. 
    Thus,
    \[
    \ol H(p)\leq \inf_{\phi\in \BUC(\R^n\times [0,\infty))} \inf S_\phi.
    \]

    Next, we prove the reverse inequality.
    Recall that $v(y,\tau)=w(y,\tau)+\overline{H}(p)\tau$ solves \eqref{cell}.
    Further, $v\in \BUC(\R^n\times [0,\infty))$.
    Hence, $\ol H(p) \in S_v$. 
    \end{proof}
    
\subsection{Further discussions}\label{further}

We first explain how our work relates to a conjecture in \cite{DG} (see also \cite{KS,P}).

\medskip

\noindent {\bf De Giorgi's conjecture.} {\it Let $F\in\Lip(\R^n\times\R,\R^n)$ be a $\mathbb Z^{n+1}$-periodic vector field, and $\varphi\in \Lip_{\rm loc}(\R^n)$. For every $\ep>0$, let $u^\ep\in \Lip_{\rm loc}(\R^n\times[0,\infty))$ be the solution of
\begin{equation*}
\begin{cases}
u^\ep_t+F(\frac{x}{\ep},\frac{t}{\ep})\cdot Du^\ep=0\quad &\text{for}\ x\in\R^n,\ t>0,
\\ u^\ep(x,0)=\varphi(x)\quad &\text{for}\ x\in\R^n.
\end{cases}
\end{equation*}
Then there exists $\overline{u}\in \Lip_{\rm loc}(\R^n\times[0,\infty))$ such that
\[\lim_{\ep\to 0}u^\ep=\overline u\]
in the weak topology of $L^p_{\rm loc}(\R^n\times[0,\infty))$, and there is a probability measure $\mu$ depending only on $F$ such that, for every $\varphi$, it follows that
\[\overline{u}(x,t)=\int_{\R^n}\varphi(x-\xi t)\, d\mu(\xi).\]
Here, for each $t>0$, if we get the convergence of the solution $y^\ep(s)$ of the associated ODE
\begin{equation*}
\begin{cases}
\dot y^\ep(s)=F\left(\frac{y^\ep}{\ep},\frac{s}{\ep}\right)\quad \text{for}\ s\in[0,t],
\\ y^\ep(t)=x,
\end{cases}
\end{equation*}
as $\ep\to 0$, and the limit has a constant velocity $\xi\in\R^n$, by the characteristic method, we can get the strong convergence of $u^\ep$.}

\medskip

Now we assume $n=1$. 
A qualitative convergence result was previously established in \cite{MMT} (see also \cite{P2}).
Although the strong convergence and convergence rate $O(\ep)$ of the solution to \eqref{transeq} are already known (see, e.g., \cite{IM2, MMT}), we present here an alternative proof as an application of our main theorems, Theorems \ref{thm1}, \ref{thm:qualitative}.

\begin{lem}
Let $F\in \Lip(\R)$ be a $\Z$-periodic function. By the classical theory of ODEs, for each $a\in \R$, we have the unique solution $y^\ep(t;a)$ of
\begin{equation}\label{ode}
\begin{cases}
\dot y^\ep(t)=F\left(\frac{y^\varepsilon(t)}{\varepsilon}\right)
\quad \text{for}\ t>0,
\\ y^\varepsilon(0)=a.
\end{cases}
\end{equation}
Then the limit of $y^\ep(t;a)$ is an affine function, and there exist $\xi\in\R$ and $C>0$ depending only on $F$ such that 
\begin{equation}\label{oderate}
|y^\ep(t;a)-(a+\xi t)|\leq C\ep \quad \text{ for } t>0.
\end{equation}
\end{lem}
\begin{proof}
In \eqref{e}, we set $H(y,r,p)=|p|^2-F(r)$, and the initial function $\varphi\equiv a\in\R$. Let $u^\ep(x,t):=y^\ep(t)$ for all $x\in\R^n$ and $t\geq 0$, by Lemma \ref{lem:CP}, it is clear that $u^\ep$ is the unique solution of \eqref{e}. Let $\overline{H}$ be the effective Hamiltonian defined in \eqref{def:eff-H}. Similarly, let $\overline{u}(x,t)=a-\overline{H}(0)t$ for all $x\in\R^n$ and $t\geq 0$, we know that $\overline{u}$ is the unique solution of \eqref{e0}. By Theorems \ref{thm1} and \ref{thm:qualitative}, we get the conclusion \eqref{oderate} with $\xi=-\overline{H}(0)$.
\end{proof}

\begin{prop}
Let $F\in \Lip(\R)$ be a $\Z$-periodic function. Assume $\varphi\in\Lip(\R)$. Let $u^\ep$ be the solution of \begin{equation}\label{transeq}
\begin{cases}
u^\ep_t+F\left(\frac{x}{\ep}\right)\cdot Du^\ep=0\quad &\text{for}\ x\in\R,\ t>0,
\\ u^\ep(x,0)=\varphi(x)\quad &\text{for}\ x\in\R.
\end{cases}
\end{equation}
there exist $\xi \in\R$ and $C>0$ depending only on $F$ such that 
\[|u^\ep(x,t)-\varphi(x-\xi t)|\leq C\|D\varphi\|_{L^\infty(\R)}\ep \quad \text{ for\ all}\ x\in\R,\ t>0.\]
\end{prop}
\begin{proof}
We first recall the characteristic method. Let $y^\ep(t;a)$ be the unique solution of \eqref{ode}. Then $a\mapsto y^\ep(t;a)$ is a Lipschitz continuous homeomorphism on $\R$ for each $t>0$. 
For each $x\in\R$ and $t>0$, let $x_0$ be the constant such that $y^\ep(t;x_0)=x$. It is quite standard to show that
\[
u^\ep(x,t)=\varphi(x_0).
\]
is the weak solution of \eqref{transeq}. Now for each $t>0$, we define $z^\ep(s)=y^\ep(t-s;x_0)$ for $s\in[0,t]$. Then $z^\ep(s)$ satisfies
\begin{equation*}
\begin{cases}
\dot z^\ep(s)=-F\left(\frac{z^\varepsilon(s)}{\varepsilon}\right)
\quad \text{for}\ s\in[0,t],
\\ z^\varepsilon(0)=x,\quad z^\ep(t)=x_0.
\end{cases}
\end{equation*}
By \eqref{oderate}, there exist $\xi\in\R$ and $C>0$ depending only on $F$ such that 
\[|u^\ep(x,t)-\varphi(x-\xi t)|=|\varphi(z^\ep(t))-\varphi(x-\xi t)|\leq C \|D\varphi\|_{L^\infty(\R)}\ep.\]
The proof is now complete.
\end{proof}

In higher-dimensional cases ($n\ge 2$), this problem becomes more intricate and may be related to the weakly coupled systems of Hamilton–Jacobi equations studied in \cite{MT}. The main difficulty is that, for $X(s)\in\R^n$ with $n\geq 2$, we do not have the comparison principle for ODEs of the form $\dot X(s)=F(X(s),s)$, for a.e. $s>0$.

\medskip

We would like to emphasize that the existence of $\overline{H}$ and $\overline{L}$ suggests a possible generalization of the Aubry--Mather theory. In \cite{M}, the author considered the smooth Tonelli Lagrangian
\[L_\omega(x,v)=L(x,v)-\omega_x(v),\quad (x,v)\in\mathbb T^n\times\R^n,\]
where $\omega$ is a closed 1-form, which can be regarded as a function of $(x,v)$ which is linear in $v$. Since $\omega$ is closed, the Euler-Lagrange flow generated by $L$ and $L_\omega$ coincide. We denote this flow by $\Phi_t$. The associated Hamiltonian is $H_\omega(x,p)=H(x,\omega+p)$. We can show that the critical value of $H_\omega$ depends only on the cohomology class $c$ of $\omega$. We minimize $A_c:=\int L_\omega\, d\mu$ over all $\Phi_t$-invariant probability measures $\mu$. The author proved that $A_c$ takes a minimum value, and denoted it by $-\alpha(c)$. It was further shown that $\alpha$ is convex and superlinear. The convex dual $\beta(h)$ of $\alpha(c)$ is the minimum value of $\int L\, d\mu$ among all $\Phi_t$-invariant probability measures $\mu$ having the rotation vector $h$. The functions $\alpha$ and $\beta$ are exactly the effective Hamiltonian and the effective Lagrangian of $H_\omega$, respectively.
Further, it was proved in \cite[Section 3]{M} that
\begin{equation}\label{beta}
    \frac{m^0(kt,kx,ky)}{kt}\to \beta\Big(\frac{x-y}{t}\Big)\quad \text{as}\quad k\to\infty,
\end{equation}
where $m^0$ is the minimal action function associated with $L$, see \eqref{m0}. 
In recent years, there have been some generalizations of the Aubry--Mather theory, see \cite{MiS,WWY3,JMT} and the references therein. 
In \cite{MiS}, the following minimizing problem was considered
\[
\min\int_{\T^n\times\R^n} \Big(L(x,v)-\lambda u_\lambda(x)\Big)\, d\mu,
\]
where $\lambda>0$, and $u_\lambda$ is the unique solution of
\[
\lambda u+H(x,Du)=0 \quad \text{for} \ x\in\T^n,
\]
and the infimum is taken over all probability measures on  $\T^n\times\R^n$ that are invariant under the flow of the discounted Euler–Lagrange equation. 
By Lemma \ref{lem:form}, we have
\[\lim_{k\to\infty}\frac{m(kt,kx,ky,0)}{kt}=\overline{L}\Big(\frac{x-y}{t}\Big),\]
which has the same form as \eqref{beta}. Thus, Theorem \ref{thm:qualitative} suggests us to consider the following problem
\[\min\int_{\T^n\times\T\times\R^n}L(x,u,v)\, d\mu,\]
where the infimum is taken over all probability measures $\mu$ on $\T^n\times\R\times\R^n$, which are invariant under the dual flow of the contact Hamiltonian flow generated by $H$, and have a given rotation vector $h\in\R^n$ in $x$-space. 
This minimizing problem seems to be a more natural generalization of the Aubry--Mather theory in our setting.

\section{Proofs of Theorems \ref{thm:holder}--\ref{thm:corrector-Holder}}\label{sec:Holder}

In this section, we assume (H1)--(H6). 
Then, the Lagrangian associated with $H(y,r,p)$ satisfies (L1)--(L4) and
\begin{itemize}
\item[(L5)] there are constants $\alpha_0,\beta_1>0$, $m_2\geq m_1>1$ such that
\begin{equation*}
\begin{cases}
    |L(y,r,v)-L(0,0,v)|\leq \alpha_0\quad &\text{for\ all}\quad (y,r,v)\in\R^n\times\R\times\R^n,
    \\  L(0,0,v)\geq \frac{1}{\beta_1}|v|^{m_1}-\beta_1\quad &\text{for\ all}\quad v\in\R^n,
    \\  |D_v L(0,0,v)|\leq \beta_1(|v|^{m_2-1}+1)\quad &\text{for a.e.}\quad v\in\R^n,
\end{cases}
\end{equation*}
where $m_1=\frac{q_2}{q_2-1}$ and $m_2=\frac{q_1}{q_1-1}$.
\end{itemize}
The first two inequalities above are easy to prove. We explain the last inequality. In this section, we denote 
\[
L^0(v):=L(0,0,v) \quad \text{ for } v\in \R^n.
\]
For $v_1,v_2\in\R^n$, we take $p_1\in\R^n$ such that $L^0(v_1)=p_1\cdot v_1-H(0,0,p_1)$, then
\[L^0(v_1)-L^0(v_2)\leq p_1\cdot v_1-H(0,0,p_1)-(p_1\cdot v_2-H(0,0,p_1))\leq |p_1||v_1-v_2|.\]
Here, we know that $v_1\in D^-H(0,0,p_1)$, where $D^-\phi(z)$ denotes the subdifferential of $\phi$ at $z$. Then, in light of (H6),
\[|v_1|\geq \frac{1}{\beta_0}|p_1|^{q_1-1}-\beta_0,\]
which implies
\[
|p_1|\leq \beta_1\left(|v_1|^{\frac{1}{q_1-1}}+1\right)
\]
for some $\beta_1>0$.
Hence,
\[L^0(v_1)-L^0(v_2)\leq \beta_1\left(|v_1|^{\frac{1}{q_1-1}}+1\right)|v_1-v_2|=\beta_1\left(|v_1|^{\frac{q_1}{q_1-1}-1}+1\right)|v_1-v_2|.\]
By symmetry, 
\[
|L^0(v_1)-L^0(v_2)|\le \beta_1\left(|v_1|^{\frac{q_1}{q_1-1}-1}+1\right)|v_1-v_2|.
\]
Dividing both sides of the inequality by $|v_1-v_2|$  and sending $v_2\to v_1$, we obtain the result. 

\begin{lem}\label{lem:Holderx}
Assume {\rm (H1)--(H6)} and $\varphi\in \BUC(\R^n)$.
There is a constant $C>0$ depending on $H$ and $\|\varphi\|_{W^{1,\infty}(\R^n)}$ such that, for all $x,y\in\R^n$ and $t>1$,
\[|u^\ep(x,t)-u^\ep(y,t)|\leq C|x-y|^{\frac{m_1}{m_1+m_2-1}}.\]
\end{lem}
\begin{proof}
Fix $x,y\in\R^n$ with $x\not=y$. 
We first show the boundedness of $|u^\ep(x,t)-u^\ep(y,t-1)|$. Let $u$ be the solution of \eqref{e0}. By the comparison principle, we have
\[-\|\varphi\|_{L^\infty(\R^n)}-\ol{H}(0)t\leq u\leq \|\varphi\|_{L^\infty(\R^n)}-\ol{H}(0)t,\]
which implies
\[|u(x,t)-u(y,t-1)|\leq 2\|\varphi\|_{L^\infty(\R^n)}+|\ol{H}(0)|
\quad \text{for\ all}\quad x,y\in\R^n,\ t>1.\]
By Theorem \ref{thm1},
\begin{equation}\label{est:u-ep}
|u^\ep(x,t)-u^\ep(y,t-1)|\leq |u(x,t)-u(y,t-1)|+2C\ep\leq 2\|\varphi\|_{L^\infty(\R^n)}+|\ol{H}(0)|+2C=:C_1    
\end{equation}
for all $x,y\in\R^n$ and $\ep\in(0,1)$.
Note that the global convergence rate $O(\ep)$ in Theorem \ref{thm1} is crucial for the estimate \eqref{est:u-ep}, as we require the bound to hold for all $t>1$.

\smallskip
We denote $d:=|x-y|$. 
Let $\gamma$ be a minimizer of $u^\ep(x,t)$ with $\gam(t)=x$. 
For all $\tau\in(0,1]$, by \eqref{est:u-ep} and (L5), we have
\begin{align*}
C_1&\geq u^\ep(x,t)-u^\ep(\gamma(t-1),t-1)=\int_{t-1}^t L\Big(\frac{\gamma(s)}{\ep},\frac{u^\ep(\gamma(s),s)}{\ep},\dot\gamma(s)\Big)\, ds
\\ &\geq \int_{t-1}^tL^0(\dot\gamma(s))\, ds-\alpha_0
\geq \int_{t-1}^t\frac{1}{\beta_1}|\dot\gamma(s)|^{m_1}\, ds-(\alpha_0+\beta_1)
\\ &\geq \int_{t-\tau}^t\frac{1}{\beta_1}|\dot\gamma(s)|^{m_1}\, ds-(\alpha_0+\beta_1)\geq 
\frac{1}{\beta_1\tau^{m_1-1}}|x-\gamma(t-\tau)|^{m_1}-(\alpha_0+\beta_1),
\end{align*}
where for the last inequality, we used the Jensen inequality. Thus,
\begin{equation}\label{x-gm}
|x-\gamma(t-\tau)|\leq C_2\tau^{1-\frac{1}{m_1}},
\end{equation}
where $C_2:=((C_1+\alpha_0+\beta_1)\beta_1)^{\frac{1}{m_1}}$. 
Let $\alpha$ be the line segment connecting $\gam(t-\tau)$ to $y$ in time $\tau$, that is, define $\alpha:[0,\tau]\to \mathbb R^n$ as
\[\alpha(s):=\gamma(t-\tau)+\frac{s}{\tau}\Big(y-\gamma(t-\tau)\Big) \quad \text{ for } s\in [0,\tau]. 
\]
Then, 
\[\dot\alpha(s)=\frac{y-\gamma(t-\tau)}{\tau}  \quad \text{ for } s\in [0,\tau].\]
Define
\begin{equation*}
\eta(s):=
\begin{cases}
\gamma(s)\quad &\text{for} \ s\in[0,t-\tau],
\\ \alpha(s-t+\tau)\quad &\text{for} \ s\in(t-\tau,t].
\end{cases}
\end{equation*} 
We have
\begin{align*}
&u^\ep(y,t)-u^\ep(x,t)\\
\leq\ & \int_0^t L\Big(\frac{\eta(s)}{\ep},\frac{u^\ep(\eta(s),s)}{\ep},\dot\eta(s)\Big)\, ds-\int_0^t L\Big(\frac{\gamma(s)}{\ep},\frac{u^\ep(\gamma(s),s)}{\ep},\dot\gamma(s)\Big)\, ds
\\ 
=\ &\int_{0}^\tau L\Big(\frac{\alpha(s)}{\ep},\frac{u^\ep(\alpha(s),s+t-\tau)}{\ep},\dot\alpha(s)\Big)\, ds-\int_{t-\tau}^t L\Big(\frac{\gamma(s)}{\ep},\frac{u^\ep(\gamma(s),s)}{\ep},\dot\gamma(s)\Big)\, ds
\\ 
\leq\ & \int_{0}^\tau L^0(\dot\alpha(s))\, ds-\int_{t-\tau}^t L^0(\dot\gamma(s))\, ds+2\alpha_0\tau.
\end{align*}
Since $L^0(v)$ is convex, by the Jensen inequality,
\[L^0\Big(\frac{x-\gamma(t-\tau)}{\tau}\Big)=L^0\Big(\frac{1}{\tau}\int_{t-\tau}^t \dot\gamma(s)\, ds\Big)\leq \frac{1}{\tau} \int_{t-\tau}^t L^0(\dot\gamma(s))\, ds,\]
which implies
\[u^\ep(y,t)-u^\ep(x,t)\leq \bigg(L^0\Big(\frac{y-\gamma(t-\tau)}{\tau}\Big)-L^0\Big(\frac{x-\gamma(t-\tau)}{\tau}\Big)\bigg)\tau+2\alpha_0\tau.\]
Note that due to (L4), (L5) with \eqref{x-gm}, we have 
\begin{align*}
&L^0\Big(\frac{y-\gamma(t-\tau)}{\tau}\Big)-L^0\Big(\frac{x-\gamma(t-\tau)}{\tau}\Big)\\    
\le&\, 
\beta_1\left(
\left|\frac{y-\gamma(t-\tau)}{\tau}\right|^{m_2-1}
+
\left|\frac{x-\gamma(t-\tau)}{\tau}\right|^{m_2-1}
+1
\right)\cdot \frac{|x-y|}{\tau}\\
\le&\, 
\beta_1\left(
\left(
\frac{|x-y|}{\tau}+C_2\tau^{-\frac{1}{m_1}}\right)^{m_2-1}
+(C_2\tau^{-\frac{1}{m_1}})^{m_2-1}+1
\right)\cdot \frac{|x-y|}{\tau}. 
\end{align*}
Take $\tau=d^\alpha$ with $\alpha\in(0,1]$ to be chosen later. 
Then, 
\begin{equation}\label{ineq:alpha}
u^\ep(y,t)-u^\ep(x,t)
\leq \beta_1\Big((d^{1-\alpha}+C_2d^{-\frac{\alpha}{m_1}})^{m_2-1}+(C_2d^{-\frac{\alpha}{m_1}})^{m_2-1}+1\Big)d+2\alpha_0 d^\alpha.
\end{equation}

We find $\alpha\in(0,1]$ to optimize inequality \eqref{ineq:alpha}. 
We first notice that we should necessarily have 
\[
-\frac{\alpha(m_2-1)}{m_1}+1>0
\quad\iff\quad 
\alpha<\frac{m_1}{m_2-1}. 
\]
Also, since $\alpha\in(0,1]$, we always have $1-\alpha>-\frac{\alpha}{m_1}$, that is, $\alpha<\frac{m_1}{m_1-1}$. 
Thus, $d^{-\frac{\alpha(m_2-1)}{m_1}+1}$ is a leading term in \eqref{ineq:alpha}. 
Therefore, we only need to maximize 
\[
\min\left\{
-\frac{\alpha(m_2-1)}{m_1}+1, \alpha\right\}
\]
with respect to $\alpha\in(0,1]$, 
which is given by $\alpha=\frac{m_1}{m_1+m_2-1}$.  
Therefore, for all $x,y\in\R^n$ and $t>1$,
\[|u^\ep(x,t)-u^\ep(y,t)|\leq C|x-y|^{\frac{m_1}{m_1+m_2-1}}.\]
\end{proof}

\begin{lem}\label{lem:holdert}
Assume {\rm (H1)--(H6)} and $\varphi\in \BUC(\R^n)$.
There is a constant $C>0$ depending on $H$ and $\|\varphi\|_{W^{1,\infty}(\R^n)}$ such that, for all $x,y\in\R^n$, $\delta>0$ and $t>\max\{1,\delta\}$, 
\[|u^\ep(x,t)-u^\ep(x,t-\delta)|\leq C\delta^{\frac{m_1}{m_1+m_2}}.\]
\end{lem}
\begin{proof}
We first show
\[u(x,t)-u(x,t-\delta)\leq C_3\delta.\]
Let $\gamma_1:[0,t-\delta]\to\R^n$ be a minimizer of $u^\ep(x,t-\delta)$. Define
\begin{equation*}
\eta_1(s):=
\begin{cases}
\gamma_1(s)\quad &\text{for} \ s\in[0,t-\delta],
\\ x\quad &\text{for} \ s\in(t-\delta,t].
\end{cases}
\end{equation*}
We get
\begin{align*}
u^\ep(x,t)&\leq \varphi(\eta_1(0))+\int_0^tL\Big(\frac{\eta_1(s)}{\ep},\frac{u^\ep(\eta_1(s),s)}{\ep},\dot\eta_1(s)\Big)\, ds
\\ &=u^\ep(x,t-\delta)+\int_{t-\delta}^tL\Big(\frac{x}{\ep},\frac{u^\ep(x,s)}{\ep},0\Big)\, ds\leq u^\ep(x,t-\delta)+C_3\delta,
\end{align*}
where $C_3:=|L^0(0)|+\alpha_0$.

We prove the other direction. 
Let $\gamma:[0,t]\to\R^n$ be a minimizer of $u^\ep(x,t)$ with $\gam(t)=x$. 
Take $\tau>\delta$ to be chosen. Define $\alpha:[0,\tau-\delta]\to \mathbb R^n$ as
\[\alpha(s):=\gamma(t-\tau)+\frac{s}{\tau-\delta}\Big(x-\gamma(t-\tau)\Big) \quad \text{ for } s\in [0,\tau-\delta],\]
then
\[\dot\alpha(s)=\frac{x-\gamma(t-\tau)}{\tau-\delta} \quad \text{ for } s\in [0,\tau-\delta].\]
Define
\begin{equation*}
\eta_2(s):=
\begin{cases}
\gamma(s)\quad &\text{for} \ s\in[0,t-\tau],
\\ \alpha(s-t+\tau)\quad &\text{for} \ s\in(t-\tau,t-\delta].
\end{cases}
\end{equation*}
Here, $\alpha$ is a line segment connecting $\gamma(t-\tau)$ to $x$ in time $\tau-\delta$.
We get
\begin{align*}
&u^\ep(x,t-\delta)-u^\ep(x,t)
\\\leq\ & \int_0^{t-\delta} L\Big(\frac{\eta_2(s)}{\ep},\frac{u^\ep(\eta_2(s),s)}{\ep},\dot\eta_2(s)\Big)\, ds-\int_0^t L\Big(\frac{\gamma(s)}{\ep},\frac{u^\ep(\gamma(s),s)}{\ep},\dot\gamma(s)\Big)\, ds
\\ =\ &\int_{0}^{\tau-\delta} L\Big(\frac{\alpha(s)}{\ep},\frac{u^\ep(\alpha(s),s+t-\tau)}{\ep},\dot\alpha(s)\Big)\, ds-\int_{t-\tau}^t L\Big(\frac{\gamma(s)}{\ep},\frac{u^\ep(\gamma(s),s)}{\ep},\dot\gamma(s)\Big)\, ds
\\ \leq \ & \int_{0}^{\tau-\delta} L^0(\dot\alpha(s))\, ds-\int_{t-\tau}^t L^0(\dot\gamma(s))\, ds+2\alpha_0\tau.
\end{align*}
By the Jensen inequality and (L5), we then get
\begin{align*}
&u^\ep(x,t-\delta)-u^\ep(x,t)
\\ \leq \ & L^0\Big(\frac{x-\gamma(t-\tau)}{\tau-\delta}\Big)(\tau-\delta)-L^0\Big(\frac{x-\gamma(t-\tau)}{\tau}\Big)\tau+2\alpha_0\tau
\\ =\ &\bigg(L^0\Big(\frac{x-\gamma(t-\tau)}{\tau-\delta}\Big)-L^0\Big(\frac{x-\gamma(t-\tau)}{\tau}\Big)\bigg)(\tau-\delta)-L^0\Big(\frac{x-\gamma(t-\tau)}{\tau}\Big)\delta+2\alpha_0\tau
\\ \leq\ & \beta_1\Big(\Big|\frac{x-\gamma(t-\tau)}{\tau-\delta}\Big|^{m_2-1}+1\Big)\Big|\frac{x-\gamma(t-\tau)}{\tau-\delta}-\frac{x-\gamma(t-\tau)}{\tau}\Big|(\tau-\delta)-\min_{v\in\R^n}L^0(v)\delta+2\alpha_0\tau
\\ \leq  \ &\beta_1\Big(\Big|\frac{x-\gamma(t-\tau)}{\tau-\delta}\Big|^{m_2-1}+1\Big)\Big|\frac{x-\gamma(t-\tau)}{\tau}\Big|\delta+|\min_{v\in\R^n}L^0(v)|\delta+2\alpha_0\tau.
\end{align*}
We take $\tau=\delta^\alpha$ with $\alpha\in(0,1)$. Note that $\delta^\alpha$ is much larger than $\delta$ when $\delta$ is small. Then by \eqref{x-gm}, we get $|x-\gamma(t-\tau)|\leq C_2\delta^{\alpha(1-m_1^{-1})}$, and
\[
u^\ep(x,t-\delta)-u^\ep(x,t)
\leq C_4\Big(\delta^{1-\frac{\alpha m_2}{m_1}}+\delta^\alpha\Big),
\]
for some $C_4>0$ depending on $H$ and $\|\varphi\|_{W^{1,\infty}(\R^n)}$. It is easy to show that the optimal choice of $\alpha$ is $\frac{m_1}{m_1+m_2}$. Then we get
\[-C_4\delta^{\frac{m_1}{m_1+m_2}}\leq u^\ep(x,t)-u^\ep(x,t-\delta)
\leq C_3\delta.\]
The proof is now complete.
\end{proof}
We are now ready to give the proofs of Theorems \ref{thm:holder}--\ref{thm:corrector-Holder}.

\begin{proof}[Proof of Theorem {\rm\ref{thm:holder}}]
    By combining Lemmas \ref{lem:Holderx} and \ref{lem:holdert}, we get the desired conclusion.
\end{proof}

\begin{proof}[Proof of Theorem {\rm\ref{thm:corrector-Holder}}]
Let $w$ be the solution of \eqref{Hp}. 
By Theorem \ref{thm:qualitative}, we have
\[|(w(y_1,\tau)+\ol{H}(p)\tau)-(w(y_2,\tau-1)+\ol{H}(p)(\tau-1))|\leq 2C
\]
for $y_1, y_2\in\R^n$ and $\tau>1$, 
which implies
\begin{equation}\label{est:w}
|w(y_1,\tau)-w(y_2,\tau-1)|\leq 2C+|\ol{H}(p)|.
\end{equation}
We emphasize that the uniform estimate \eqref{est:w} is essentially due to the quantitative result of Theorem \ref{thm1}, which is important to obtain the H\"older estimate in the proof Lemma \ref{lem:Holderx}. 
Note that the Lagrangian associated with $H_p(y,r,q):=H(y,p\cdot y+r,p+q)$ is
\[L_p(y,r,v)=L(y,p\cdot y+r,v)-p\cdot v.\]
It is easy to check that $L_p$ satisfies (L5). 
By repeating the proofs of Lemmas \ref{lem:Holderx} and \ref{lem:holdert} for the Lagrangian $L_p$, we get the H\"older estimate for $w(y,\tau)$, 
\[
|w(y_1,\tau_1)-w(y_2,\tau_2)|\leq C\Big(|y_1-y_2|^{\frac{m_1}{m_1+m_2-1}}+|\tau_1-\tau_2|^{\frac{m_1}{m_1+m_2}}\Big),
\]
for all $y_1,y_2\in\R^n$, $\tau_1,\tau_2>1$, with the constant $C$ depending only on $H$, $p$, and $\|w_0\|_{L^{\infty}(\R^n)}$.
By repeating the proof of Lemma \ref{lem:u-optimal-control}, $w$ is Lipschitz on $\R^n\times [0,2]$.

Hence, the corrector
\[v(y,\tau):=w(y,\tau)+\ol{H}(p)\tau\]
is uniformly bounded and equi-H\"older continuous. 
Set
\[
v^k(y,\tau) = v(y,\tau+k) \quad \text{ for } (y,\tau) \in \R^n \times [-k+2,\infty).
\]
By using the Arzel\`a--Ascoli theorem and a diagonal argument, and by passing to a subsequence if needed, we imply that $\{v^k\}$ converges to $\bar v$ locally uniformly as $k\to\infty$.
Then, $\bar v$ is bounded and H\"older continuous with
\[
|\bar v(y_1,\tau_1)-\bar v(y_2,\tau_2)|\leq C\Big(|y_1-y_2|^{\frac{m_1}{m_1+m_2-1}}+|\tau_1-\tau_2|^{\frac{m_1}{m_1+m_2}}\Big)
\]
for all $y_1,y_2\in\R^n$, $\tau_1,\tau_2\in \R$.
Further, by the stability of viscosity solutions, $\bar v$ satisfies
\begin{equation*}
\bar v_\tau+H(y,p\cdot y+\bar v-\overline{H}(p)\tau,p+D_y \bar v)=\overline{H}(p)\quad \text{ for } (y,\tau) \in\R^n\times \R, 
\end{equation*}
which completes the proof.
\end{proof}

\begin{rem}
    We note that our H\"older regularity result in Theorem \ref{thm:holder} is also applicable for more general cases of the form: 
    \[
    u_t+H(x,t,u,Du)=0 \quad \text{ for } x\in \R^n,\ t>0.
    \]
    Here, we first assume that $H(x,t,u,p)$ is uniformly Lipschitz continuous in $u$, locally Lipschitz continuous in $t$, convex and superlinear in $p$. These structural conditions are imposed to ensure the well-posedness of the problem and the validity of the optimal control representation formula. In addition, similar to (H6), we further assume that there are constants $C>0$, $q_2\geq q_1>1$ such that
    \[
    \begin{cases}
        |H(x,t,u,p)-H(0,0,0,p)| \leq C \quad &\text{ for all $(x,t,u)\in\R^n\times\R\times\R$},\\
        H(0,0,0,p)\leq C(|p|^{q_2}+1)\quad &\text{ for all $p\in\R^n$},\\
        |D_pH(0,0,0,p)|\geq \frac{1}{C}|p|^{q_1-1}-C\quad &\text{ for all $p\in\R^n$},\\
        |u(x,t)-u(y,t-1)| \leq C \quad &\text{ for all $x,y\in\R^n$, and $t>1$.}
    \end{cases} 
    \]
\end{rem}

\appendix

\section{The comparison principle for \eqref{eq:ini-1}--\eqref{eq:ini-2}}\label{sec:conv}

\begin{lem} \label{lem:CP}
Assume {\rm (H1), (H2), (H5)}. 
Let $u,v \in \AL(\R^n\times[0,T])$ be an upper semicontinuous subsolution and a lower semi-continuous supersolution of \eqref{eq:ini-1}--\eqref{eq:ini-2} on $\R^n\times [0,T]$, respectively.
Assume further that either $u$ or $v$ is Lipschitz continuous. 
Then, 
\[
u\leq v\quad \text{ on $\R^n\times[0,T]$.}
\]
\end{lem}

\begin{proof}
Our proof is similar to that of \cite[Proposition 1]{IM}, and we give it for completeness of the paper. 

Set $\tilde u=e^{-K t}u$ and $\tilde v=e^{-K t}v$. It is clear that $\tilde u$ and $\tilde v$ are respectively a subsolution and a supersolution of
 \[w_t+Kw+e^{-K t}H(x,e^{K t}w,e^{K t}Dw)=0.\]
It suffices to prove $\tilde u\leq \tilde v$.
We divide the proof into several steps.

\medskip

\noindent {\bf Step 1.} We show that there exists $\Theta>0$ such that
\[\tilde u(x,t)-\tilde v(y,t)\leq \Theta(1+|x-y|).\]
To obtain this, it suffices to obtain
\[\tilde u(x,t)-\tilde v(y,t)\leq \Theta\zeta(x-y),\]
where $\zeta(z)=\sqrt{1+|z|^2}$. Since $u$ and $v$ are at most of linear growth, there is $L>0$ depending on $T$ such that
\begin{equation}\label{L}
\tilde u(x,t)-\tilde v(y,s)\leq L(1+|x|+|y|)\leq L(1+2|x|+|y-x|).
\end{equation}
We consider a family of functions $\{\beta_R\}\subset C^2(\mathbb R^n)$ parameterized by $R\geq 1$ as introduced in \cite[Theorem 5.1]{CIL}
\begin{equation*}
\begin{cases}
\beta_R\geq 0,
\\ \liminf_{|x|\to \infty}\beta_R(x)/|x|\geq 3L,
\\ |D\beta_R|\leq C,
\\ \lim_{R\to \infty}\beta_R(x)=0\qquad \text{ for } x\in\R^n.
\end{cases}
\end{equation*}
Here, the constant $C>0$ is independent of $R$.
For $\mu>0$, we define
\[M_{\Theta}:=\sup_{t,s\in(0,T],\,x,y\in\mathbb R^n}\Big\{\tilde u(x,t)-\tilde v(y,t)-\Theta e^{\mu t}\zeta(x-y)-\beta_R(x)\Big\}.\]
and
\[M_{\Theta,\nu}:=\sup_{t,s\in(0,T],\,x,y\in\mathbb R^n}\Big\{\tilde u(x,t)-\tilde v(y,s)-\frac{(s-t)^2}{2\nu}-\Theta e^{\mu t}\zeta(x-y)-\beta_R(x)\Big\}.\]
We want to show that there is $\Theta>0$ such that $M_{\Theta}\leq 0$ for all $R\geq 1$. Then we can conclude by letting $R\to\infty$. We argue by contradiction. Assume that there is $\Theta_0>0$ large such that for all $\Theta\ge \Theta_0$, there is $R:=R(\Theta)\geq 1$ such that $M_{\Theta}>0$. Then we have $M_{\Theta,\nu}\geq M_\Theta>0$. In view of \eqref{L} and the construction of $\beta_R$, when $\Theta\geq 2L$, the supremum in $M_{\Theta,\nu}$ is attained at $(\bar t,\bar s,\bar x,\bar y)$ and  $(\bar t,\bar s,\bar x,\bar y)\to (\tilde t,\tilde t,\tilde x,\tilde y)$ as $\nu\to 0$. Moreover, $(\tilde t,\tilde x,\tilde y)$ is a point that realizes $M_\Theta$.

Assume $\tilde t=0$. 
As $M_{\Theta}>0$ and $\tilde x\neq \tilde y$, 
\[\varphi(\tilde x)-\varphi(\tilde y)-\Theta|\tilde x-\tilde y|\geq \tilde u(\tilde x,0)-\tilde v(\tilde y,0)-\Theta|\tilde x-\tilde y|>0.
\]
This leads to a contradiction if $\Theta$ is large enough since $\varphi$ is Lipschitz continuous. Thus, $\tilde t>0$. 
Hence, we have $\bar t,\bar s>0$ for all $\nu>0$ small enough. We have viscosity inequalities
\begin{align*}
\mu \Theta e^{\mu\bar t}\zeta(\bar x-\bar y)+\frac{\bar t-\bar s}{\nu}+K\tilde u(\bar x,\bar t)+e^{-K \bar t}H(\bar x,e^{K \bar t}\tilde u(\bar x,\bar t),e^{K \bar t}(\bar p+D\beta_R(\bar x)))\leq 0
\end{align*}
and
\begin{align*}
\frac{\bar t-\bar s}{\nu}+K \tilde v(\bar y,\bar s)+e^{-K\bar s}H(\bar y,e^{K \bar s}v(\bar y,\bar s),e^{K \bar s}\bar p)\geq 0,
\end{align*}
where $\bar p=\Theta e^{\mu \bar t}D\zeta(\bar x-\bar y)$. 
Taking the difference of the two inequalities yields
\begin{equation}\label{ux}
\begin{aligned}
\mu \Theta e^{\mu\bar t}\zeta(\bar x-\bar y)\leq &-K(\tilde u(\bar x,\bar t)-\tilde v(\bar y,\bar s))
\\ &-e^{-K \bar t}H(\bar x,e^{K \bar t}\tilde u(\bar x,\bar t),e^{K \bar t}(\bar p+D\beta_R(\bar x)))
\\ &+e^{-K \bar s}H(\bar y,e^{K \bar s}\tilde v(\bar y,\bar s),e^{K \bar s}\bar p).
\end{aligned}
\end{equation}
We deduce from $M_{\Theta,\nu}>0$ that $\tilde u(\bar x,\bar t)\geq \tilde v(\bar y,\bar s)$. Then
\begin{align*}
&-K(\tilde u(\bar x,\bar t)-\tilde v(\bar y,\bar s))
\\ &-e^{-K \bar t}H(\bar x,e^{K \bar t}\tilde u(\bar x,\bar t),e^{K \bar t}(\bar p+D\beta_R(\bar x)))+e^{-K \bar t}H(\bar x,e^{K \bar t}\tilde v(\bar y,\bar s),e^{K \bar t}(\bar p+D\beta_R(\bar x)))\leq 0.
\end{align*}
Without loss of generality, we assume that $\tilde v$ is Lipschitz continuous. We show that $\bar p$ is bounded. Taking $x=y=\bar x$, $t=\bar t$ and $s=\bar s$, we get
\[\Theta e^{\mu t}\zeta(\bar x-\bar y)\leq \tilde v(\bar x,\bar s)-\tilde v(\bar y,\bar s)+\Theta e^{\mu t}\leq \|D\tilde v\|_{L^\infty}|\bar x-\bar y|+\Theta e^{\mu t}.\]
We take $\Theta$ large, then we get
\[(\Theta e^{\mu t})^2(1+|\bar x-\bar y|^2)\leq (\|D\tilde v\|_{L^\infty}|\bar x-\bar y|+\Theta e^{\mu t})^2,\]
which implies
\[|\bar x-\bar y|\leq \frac{2\Theta e^{\mu t}\|D\tilde v\|_{L^\infty}}{(\Theta e^{\mu t})^2-\|D\tilde v\|_{L^\infty}^2}.\]
Calculating directly, we have
\[|\bar p|=\Theta e^{\mu \bar t}\frac{|\bar x-\bar y|}{\sqrt{1+|\bar x-\bar y|^2}}\leq \Theta e^{\mu \bar t}|\bar x-\bar y|\leq \frac{2\|D\tilde v\|_{L^\infty}}{1-(\frac{\|D\tilde v\|_{L^\infty}}{\Theta e^{\mu \bar t}})^2}\leq 4\|D\tilde v\|_{L^\infty},\]
where we choose $\Theta>0$ large enough such that $(\frac{\|D\tilde v\|_{L^\infty}}{\Theta e^{\mu \bar t}})^2\leq \frac{1}{2}$. Then \[-e^{-K \bar t}H(\bar x,e^{K \bar t}\tilde v(\bar y,\bar s),e^{K \bar t}(\bar p+D\beta_R(\bar x)))+e^{-K \bar s}H(\bar y,e^{K \bar s}\tilde v(\bar y,\bar s),e^{K \bar s}\bar p)\]
is bounded by a constant independent of $\Theta$.

We conclude that the right-hand side of \eqref{ux} is independent of $\Theta$, while the left-hand side satisfies $\mu \Theta e^{\mu\bar t}\zeta(\bar x-\bar y)\geq \mu \Theta$. Choosing $\Theta$ large enough we get a contradiction.

\medskip

\noindent {\bf Step 2.} The remaining proof is standard. It suffices to show that $\tilde u(x,t)-\eta t-\frac{\alpha}{2}|x|^2\leq \tilde v(x,t)$ for all $\eta>0$ and all small $\alpha>0$. We then conclude by letting $\eta,\alpha\to 0$ for each $(x,t)\in\R^n\times(0,T]$. We argue by contradiction. Assume
\begin{align*}
    M&:=\sup_{(x,t)\in\R^n\times(0,T]}\Big(\tilde u(x,t)-\eta t-\frac{\alpha}{2}|x|^2-\tilde v(x,t)\Big)
    \\ &=\tilde u(x_0,t_0)-\eta t_0-\frac{\alpha}{2}|x_0|^2-\tilde v(x_0,t_0)>0.
\end{align*}
Here, since $\tilde u$ and $\tilde v$ have at most linear growth, the above maximum can be achieved. Let $\nu,\kappa>0$. Define
\[M_{\alpha,\kappa,\nu}:=\sup_{t,s\in(0,T],\,x,y\in\mathbb R^n}\Big\{\tilde u(x,t)-\tilde v(y,s)-\frac{(s-t)^2}{2\nu}-\frac{|x-y|^2}{2\kappa}-\frac{\alpha}{2}|x|^2-\eta t\Big\}.\]
We have $M_{\alpha,\kappa,\nu}\geq M>0$. Since
\[\tilde u(x,t)-\tilde v(y,s)\leq L(1+2|x|+|y-x|),\]
we know that the supremum is attained at $(\bar t,\bar s,\bar x,\bar y)$. Also, we have
\[
\frac{(\bar s-\bar t)^2}{2\nu}\leq L(1+2|\bar x|+|\bar y-\bar x|)-\frac{|\bar x-\bar y|^2}{2\kappa}-\frac{\alpha}{2}|\bar x|^2
\leq L+\frac{\kappa L^2}{2}+\frac{2L^2}{\alpha},
\]
which implies that $|\bar t-\bar s|\to 0$ as $\nu\to 0$ for fixed $\alpha,\kappa$. Thus, $(\bar t,\bar s,\bar x,\bar y)$ is close to $(\tilde t,\tilde t,\tilde x,\tilde y)$ when $\nu$ is small. From the first step, since $\tilde u-\tilde v$ is upper semi-continuous, we have
\[\tilde u(\bar x,\bar t)-\tilde v(\bar y,\bar s)\leq \tilde u(\tilde x,\tilde t)-\tilde v(\tilde y,\tilde t)+1\leq \Theta(1+|\tilde x-\tilde y|)+1\leq \Theta(1+|\bar x-\bar y|)+2\]
for $\nu$ small enough. 
Then by taking a larger $\Theta$, we have 
\[\tilde u(\bar x,\bar t)-\tilde v(\bar y,\bar s)\leq \Theta(1+|\bar x-\bar y|).\]
From $M_{\alpha,\kappa,\nu}>0$, we have
\begin{align*}
\frac{\alpha}{2}|\bar x|^2&\leq \tilde u(\bar x,\bar t)-\tilde v(\bar y,\bar s)-\frac{|\bar x-\bar y|^2}{2\kappa}
\\ &\leq \Theta(1+|\bar x-\bar y|)-\frac{|\bar x-\bar y|^2}{2\kappa}\leq \Theta+\frac{\kappa \Theta^2}{2}.
\end{align*}
We conclude that $\alpha|\bar x|\to 0$ as $\alpha\to 0$ for small $\kappa$. Also,
\begin{equation*}
\frac{|\bar x-\bar y|^2}{2\kappa}\leq \Theta(1+|\bar x-\bar y|)\leq \Theta+\frac{|\bar x-\bar y|^2}{4\kappa}+\kappa \Theta^2,
\end{equation*}
which implies that $|\bar x-\bar y|^2/\kappa$ is bounded, then $\bar x$ and $\bar y$ have the same limit as $\kappa\to 0$. Moreover, letting $\nu,\kappa \to 0$ and using the fact that
\[M\leq \liminf\Big(\tilde u(\bar x,\bar t)-\tilde v(\bar y,\bar s)-\eta \bar t-\frac{\alpha}{2}|\bar x|^2\Big)\leq \limsup\Big(\tilde u(\bar x,\bar t)-\tilde v(\bar y,\bar s)-\eta\bar t-\frac{\alpha}{2}|\bar x|^2\Big)\leq M,\]
where we used $M\leq M_{\alpha,\kappa,\nu}$ in the first inequality, and we used the fact that $(\bar x,\bar t)$ and $(\bar y,\bar s)$ have the same limit as $\nu,\kappa \to 0$ in the last inequality. We get \[M=\lim_{\nu,\kappa\to 0} \Big(\tilde u(\bar x,\bar t)-\tilde v(\bar y,\bar s)-\eta \bar t-\frac{\alpha}{2}|\bar x|^2\Big).\] 
By
\[M:=\sup\Big(\tilde u-\tilde v-\eta t-\frac{\alpha}{2}|\bar x|^2\Big)\leq M_{\alpha,\kappa,\nu}\leq \tilde u(\bar x,\bar t)-\tilde v(\bar y,\bar s)-\eta t-\frac{\alpha}{2}|\bar x|^2,\]
we get $M_{\alpha,\kappa,\nu}\to M$ as $\nu,\kappa \to 0$, which implies that
\[\frac{(\bar s-\bar t)^2}{\nu}+\frac{|\bar x-\bar y|^2}{\kappa}\to 0.\]
Since $M>0$, the initial condition ensures that $t_0>0$, which implies that $\bar t,\bar s>0$. Writing both viscosity inequalities and subtracting them yields
\begin{equation}\label{eta<}
\begin{aligned}
\eta\leq &-K(\tilde u(\bar x,\bar t)-\tilde v(\bar y,\bar s))
\\ &-e^{-K\bar t}H(\bar x,e^{K \bar t}\tilde u(\bar x,\bar t),e^{K \bar t}(\bar p+\alpha \bar x))
+e^{-K \bar s}H(\bar y,e^{K \bar s}\tilde v(\bar y,\bar s),e^{K \bar s}\bar p),
\end{aligned}
\end{equation}
where $\bar p=(\bar x-\bar y)/\kappa$.

Without loss of generality, we assume that $\tilde v$ is Lipschitz continuous. We show that $\bar p$ is bounded. Taking $x=y=\bar x$, $t=\bar t$ and $s=\bar s$, we get
\[\tilde u(\bar x,\bar t)-\tilde v(\bar x,\bar s)\leq \tilde u(\bar x,\bar t)-\tilde v(\bar y,\bar s)-\frac{|\bar x-\bar y|^2}{2\kappa},\]
which implies
\[\frac{|\bar x-\bar y|^2}{2\kappa}\leq \tilde v(\bar x,\bar s)-\tilde v(\bar y,\bar s)\leq \|D\tilde v\|_{L^\infty}|\bar x-\bar y|.\]
Thus, $|\bar p|\leq 2\|D\tilde v\|_{L^\infty}$.

By $M_{\alpha,\kappa,\nu}>0$, we have $\tilde u(\bar x,\bar t)-\tilde v(\bar y,\bar s)>0$. Then
\begin{equation*}
\begin{aligned}
-K(\tilde u(\bar x,\bar t)-\tilde v(\bar y,\bar s))&-e^{-K\bar t}H(\bar x,e^{K \bar t}\tilde u(\bar x,\bar t),e^{K \bar t}(\bar p+\alpha \bar x))
\\ &+e^{-K\bar t}H(\bar x,e^{K \bar t}\tilde  v(\bar y,\bar s),e^{K \bar t}(\bar p+\alpha \bar x))
\leq 0.
\end{aligned}
\end{equation*}
Note that $\alpha|\bar x|\to 0$ as $\alpha\to 0$ and $\bar p$ is bounded. Since we assume (H1), by the continuity of $H$, we can take $\alpha$ small so that
\[
-e^{-K\bar t}H(\bar x,e^{K \bar t}\tilde  v(\bar y,\bar s),e^{K \bar t}(\bar p+\alpha \bar x))+e^{-K\bar t}H(\bar x,e^{K \bar t}\tilde  v(\bar y,\bar s),e^{K \bar t}\bar p)
\]
is arbitrarily small. 
For fixed $\alpha$, $\bar x$ is bounded. Then $\bar y$ is bounded, since $|\bar x-\bar y|\to 0$ as $\kappa\to 0$. Since $\tilde v(y,s)$ has at most linear growth in $y$, $\tilde v(\bar y,\bar s)$ is bounded. By the continuity of $H$, since $\bar p$ is bounded, for fixed $\alpha,\kappa$, we choose $\nu$ small enough so that
\[
-e^{-K\bar t}H(\bar x,e^{K \bar t}\tilde  v(\bar y,\bar s),e^{K \bar t}\bar p)+e^{-K\bar s}H(\bar x,e^{K \bar s}\tilde  v(\bar y,\bar s),e^{K \bar s}\bar p)\]
is arbitrarily small. 
By taking $\kappa$ small, we also have that
\[
-e^{-K\bar s}H(\bar x,e^{K \bar s}\tilde  v(\bar y,\bar s),e^{K \bar s}\bar p)+e^{-K\bar s}H(\bar y,e^{K \bar s}\tilde  v(\bar y,\bar s),e^{K \bar s}\bar p)\]
is arbitrarily small. Finally, by first taking $\nu\ll \kappa,\alpha$, and then letting $\kappa,\alpha$ be sufficiently small, we conclude that the right-hand side of \eqref{eta<} can be arbitrarily small. This leads to a contradiction.
\end{proof}

\section{Derivation of \eqref{e}}\label{sec:moti}

In this section, we give a derivation of convex Hamilton-Jacobi equations with $(u/\ep)$-periodic Hamiltonians. 
We repeat some arguments in \cite{IMR} with careful adaptation to our settings. 
Dislocations which we would describe in this paper are line defects in crystals. 
Consider the situation that there are $N$ dislocation lines with the same Burgers vector, and all are contained in a single slip plane. 

Since we are interested in the effective dynamics for a collection of dislocation lines, we introduce a small parameter $\ep>0$ to describe the distance scale of each dislocation line. 
We use the level-set approach to describe dislocation lines, and 
therefore consider an auxiliary smooth function $v:\R^n\times[0,\infty)\to\R$ so that the $j$-th dislocation for $j\in\{1,\ldots, N\}$ is described by 
\[
\{x\in\R^n\mid v(x,t)=\ep j\}.
\]
The dislocation lines are affected by several obstacles in real life such as inclusions, precipitates, and 
all these obstacles are described by a positive $\ep\Z^n$-periodic function 
$c(\frac{x}{\ep})$, 
which is called the \textit{external force}. 
Moreover, we consider an inherent natural force pushing outward, 
the so-called \textit{Peach-Koehler} force at the position $x\in\R^n$ created by the $j$-th dislocation. This force is described by 
\[
c_0\ast  \mathbf{1}_{\{v(\cdot,t)>j\ep\}}(x), 
\]
where we denote by $\mathbf{1}_A$ the characteristic function of the set $A\subset\R^n$, i.e., 
$\mathbf{1}_A(x)=1$ for $x\in A$, 
$\mathbf{1}_A(x)=0$ for $x\in \R^n\setminus A$, 
and $c_0$ is a given kernel. 
Here, following to \cite{IMR}, we consider the kernel $c_0$ given by 
\[
c_0:=J-\delta_0, 
\]
where $J:\R^n\to\R$ is a given nonnegative continuous function 
satisfying $\int_{\R^n}J(z)\, dz=1$, 
which is given from physical quantities, 
and $\delta_0$ is a Dirac mass where we formally set  
\begin{equation*}
(\delta_0\ast \textbf{1}_{\{v(\cdot,t)>j\ep\}})(x):=
\begin{cases}
1 & \quad \text{if}\ v(x,t)>j\ep,
\\ \frac{1}{2} & \quad\text{if}\ v(x,t)=j\ep,
\\ 0 & \quad \text{if}\ v(x,t)<j\ep.
\end{cases}
\end{equation*}

Now, let us assume that there are $N_1,N_2\in\N$ such that 
\[
\Big(-N_1-\frac{1}{2}\Big)\ep<v(x,t)<\Big(N_2+\frac{1}{2}\Big)\ep
\quad\text{for all} \ (x,t)\in\R^n\times[0,\infty). 
\]
Then, the Peach-Koehler force at the point $x$ on the dislocation represented by the index $k$ (i.e., $v=\ep k$) created by dislocations represented by $j$ for $j=-N_1,\dots,N_2$ is given by
\begin{align*}
\bigg((J-\delta_0)\ast
\sum_{j=-N_1}^{N_2}\textbf{1}_{\{v>j\ep\}}\bigg)(x)
=&\, 
J\ast\bigg(\Big(\Big\lceil \frac{v}{\ep}-\frac{v(x)}{\ep}\Big\rceil-\frac{1}{2}\Big)\bigg)(x)\\
=&\, 
-\frac{v}{\ep}+J\ast E\Big(\frac{v}{\ep}\Big),
\end{align*}
where we denote by $\lceil \cdot \rceil$ the ceiling function, 
that is, $\lceil x\rceil$ is the smallest integer greater than or equal to $x\in\R$, 
and we set  
\[
E(r):=
k+\frac{1}{2}\quad \text{if} \ r\in(k,k+1]\ \text{with}\ k\in\mathbb Z.
\]
Noting that our external force is an obstacle to push dislocation lines inwards, we obtain the level set equation of the form 
\begin{equation}\label{vt}
v_t=\left(-c\Big(\frac{x}{\ep}\Big)-\frac{v}{\ep}+J\ast E\Big(\frac{v}{\ep}\Big)\right)|Dv|.
\end{equation}

According to \cite{IMR}, we approximate $E$ by a smooth function $E_\delta$ for $\delta>0$ satisfying
\begin{equation*}
E_\delta(\cdot+1)=E_\delta(\cdot)+1,\quad 0<\delta\le E'_\delta\le \frac{1}{\delta},
\end{equation*}
where we denote by $E_\delta'$ the derivative of $E_\delta$, 
and
\[
E_\delta(j)=j,\quad \text{and}\quad E_\delta\Big(j+\frac{1}{2}\Big)=j+\frac{1}{2}\ \text{for}\ j\in\mathbb Z.
\]
Then, 
$u^\ep:=\ep E_\delta(\frac{v}{\ep})$ satisfies
\begin{equation}\label{ee}
u^\ep_t=\left(-c\left(\frac{x}{\ep}\right)-\frac{u^\ep}{\ep}+J\ast \frac{u^\ep}{\ep}\right)|Du^\ep|
+\left(\frac{u^\ep}{\ep}-E^{-1}_\delta\left(\frac{u^\ep}{\ep}\right)\right)|Du^\ep|. 
\end{equation}
Note that setting 
$f(r,p):=\big(r-E^{-1}_\delta(r)\big)|p|$, 
we can check that 
$r\mapsto f(r,p)$ 
is $1$-periodic in $r$, and Lipschitz continuous for fixed $\delta$ and $p\in\R^n$.
Equation \eqref{ee} is a nonlocal first-order Hamilton--Jacobi equation which was studied in \cite{IMR}. See also \cite{AHLM,FIM} for the study of \eqref{vt}. We also refer the reader to \cite{PS,PV} for homogenization of PDEs depending periodically on $u/\ep$ with an anisotropic L\'evy operator of order 1.

Now we focus on the homogeneous kernel $J$, i.e., 
$J\equiv 1$. 
Then, it turns out that \eqref{ee} becomes a local Hamilton--Jacobi equation which depends on $\frac{u}{\ep}$ of the form: 
\begin{equation}\label{eq:local-dislocation}
u^\ep_t+c\left(\frac{x}{\ep}\right)|Du^\ep|
-\left(\frac{u^\ep}{\ep}-E^{-1}_\delta\left(\frac{u^\ep}{\ep}\right)\right)|Du^\ep|=0.  
\end{equation} 
In general, this is a noncoercive and nonconvex Hamilton--Jacobi equation. 
By noting that $|r-E^{-1}_\delta(r)|\le \frac{1}{2}$, 
we add an assumption 
\[
c(y)>1/2 \quad\text{for} \ y\in\R^n
\]
to make \eqref{eq:local-dislocation} coercive and convex.
From the viewpoint of a physical model, the above assumption 
is reasonable since 
the external force is much stronger than the Peach-Koehler force in general. 
However, since \eqref{eq:local-dislocation} is not uniformly Lipschitz continuous in $u^\ep$, we do not have well-posedness. Therefore, in the same spirit as in \cite[Section 2.1]{IMR}, 
we modify \eqref{eq:local-dislocation} to satisfy our assumptions 
(H2)--(H4). 
Set 
\begin{equation*}
H(y,r,p):=
\begin{cases}
c(y)|p|-(r-E^{-1}_\delta(r))|p|
\quad \text{for}\ y\in\R^n,\ r\in\R,\ |p|\le R,
\\ 
\frac{|p|^2}{2}-\frac{R^2}{2}+c(y)R-(r-E^{-1}_\delta(r))R
\quad \text{for} \ y\in\R^n,\ r\in\R,\ |p|> R, 
\end{cases}
\end{equation*}
where $R>0$ is a given large constant. 
Then, if $c(y)$ is continuous and $\Z^n$-periodic, for suitable initial data, 
it is clear that (H1)--(H6) hold. 
With this modified Hamiltonian, we propose the Hamilton--Jacobi equation 
\[
u^\ep_t+H\Big(\frac{x}{\ep},\frac{u^\ep}{\ep},Du^\ep\Big)=0
\]
as a model equation for the dislocation dynamics.

\section*{Acknowledgements}
The work of HM was partially supported by the JSPS grants: KAKENHI
\#21H04431, \#22K03382, \#24K00531. 
The work of HVT is partially supported by NSF grant DMS-2348305.

\section*{Declarations}

\noindent {\bf Conflict of interest statement:} The authors state that there is no conflict of interest.

\medskip

\noindent {\bf Data availability statement:} Data sharing not applicable to this article as no datasets were generated or analyzed during the current study.

\end{document}